\documentclass[a4paper]{amsart}                

\usepackage[latin1]{inputenc}                  
\usepackage[T1]{fontenc}                       
\usepackage[spanish,english]{babel}            
\usepackage[pdftex]{graphicx}                  
\usepackage{amsmath,amssymb,amsthm}            
\usepackage{latexsym}                          
\usepackage{delarray}                          
\usepackage{bbm}                               
\usepackage{hyperref}                          
\usepackage[pdftex,usenames,dvipsnames]{color} 
\usepackage{bbding,trfsigns}                   
\usepackage{wasysym}                           
\usepackage{datetime}                          

\DeclareMathAlphabet{\mathpzc}{OT1}{pzc}{m}{it} 

\newtheorem{Th}{Theorem}[section]              
\newtheorem{Cor}{Corollary}[section]

\newtheorem{Prop}{Proposition}[section]
\newtheorem{Lem}{Lemma}[section]

\title[Hankel multipliers of Laplace transform type]{Hankel multipliers of Laplace transform type}

\author[J.J. Betancor]{J.J. Betancor}
\address{Departamento de Análisis Matemático\\
Universidad de la Laguna\\
Campus de Anchieta, Avda. Astrofísico Francisco Sánchez, s/n\\
38271 La Laguna (Sta. Cruz de Tenerife), Spain}
\email{jbetanco@ull.es}

\author[A.J. Castro]{A.J. Castro}
\address{Departamento de Análisis Matemático\\
Universidad de la Laguna\\
Campus de Anchieta, Avda. Astrofísico Francisco Sánchez, s/n\\
38271 La Laguna (Sta. Cruz de Tenerife), Spain}
\email{alecas000@gmail.com}

\author[J. Curbelo]{J. Curbelo}
\address{Instituto de Ciencias Matemáticas\\
Consejo Superior de Investigaciones Científicas\\
Serrano 121\\
28006 Madrid, Spain} \email{jezabel.curbelo@icmat.es}

\thanks{This paper is partially supported by MTM2007/65609.}

\begin{document}

  \maketitle                                  

  \begin{abstract}
    In this paper we prove that the Hankel multipliers of Laplace transform type on $(0,\infty)^n$ are of weak type $(1,1)$. Also we analyze $L^p$-boundedness properties
    for the imaginary powers of Bessel operator on $(0,\infty)^n.$
  \end{abstract}

  \section{Introduction}

   We consider the Hankel transformation on $(0,\infty)^n$ defined by
   $$ h_{\lambda_1,\ldots,\lambda_n}(f)(x)= \int_{(0,\infty)^n}\prod_{j=1}^n (x_jy_j)^{-\lambda_j+1/2}J_{\lambda_j-1/2}(x_jy_j) f(y) \prod_{j=1}^n y_j^{2\lambda_j}dy
      , \quad  x\in(0,\infty)^n,$$
   where $\lambda_j>-1/2,$ $j=1,\ldots, n.$ Here $J_\nu$ denotes the Bessel function of the first kind and order $\nu.$ $h_{\lambda_1, \ldots, \lambda_n}$ maps
   $L^1((0,\infty)^n, \overset{n}{\underset{j=1}{\prod}} x_j^{2\lambda_j}dx)$ into $L^\infty((0,\infty)^n, \overset{n}{\underset{j=1}{\prod}} x_j^{2\lambda_j }dx)$
   ($=L^\infty((0,\infty)^n, dx)$) and it can be extended as an isometry to $L^2((0,\infty)^n, \overset{n}{\underset{j=1}{\prod}} x_j^{2\lambda_j }dx)$ being
   $h_{\lambda_1,\ldots,\lambda_n}^{-1}= h_{\lambda_1,\ldots,\lambda_n}$ on $L^2((0,\infty)^n, \overset{n}{\underset{j=1}{\prod}} x_j^{2\lambda_j }dx)$ (see \cite{BS}).\\

   If $m\in L^\infty((0,\infty)^n,dx)$ the multiplier operator $T_{\lambda_1,\ldots,\lambda_n}^m$ associated with the Hankel transformation
   $h_{\lambda_1, \ldots, \lambda_n}$ is defined by
   $$ T_{\lambda_1,\ldots,\lambda_n}^m(f)=h_{\lambda_1, \dots, \lambda_n}(mh_{\lambda_1,\ldots,\lambda_n}(f)), \quad f\in L^2((0,\infty)^n, \overset{n}{\underset{j=1}{\prod}} x_j^{2\lambda_j }dx).$$
   It is clear that $T_{\lambda_1,\ldots,\lambda_n}^m$ is bounded from $L^2((0,\infty)^n, \overset{n}{\underset{j=1}{\prod}} x_j^{2\lambda_j }dx)$ into itself.
   $L^p$-boundedness properties for Hankel multipliers have been studied in \cite{BMR}, \cite{BR}, \cite{GS}, \cite{GTre}, \cite{GSt}, \cite{Gu}, \cite{Ka}, \cite{NS} and \cite{STre}, amongst others.\\

   In this paper, following the ideas presented in \cite{GMMST} and in \cite{St1} we study Hankel multiplier operator when $m$ is of Laplace transform type. For every $\lambda>-1/2$
   we denote by $\Delta_{\lambda,x}$ the Bessel operator
   $$ \Delta_{\lambda,x}= -x^{-2\lambda}\frac{d}{dx}\left(x^{2\lambda}\frac{d}{dx}\right)=-\left( \frac{d^2}{dx^2}+ \frac{2\lambda}{x}\frac{d}{dx}\right),$$
   and if $\lambda_j>-1/2,$ $j=1,\ldots,n,$ we write
   $$ \Delta_{\lambda_1,\ldots,\lambda_n}= \sum_{j=1}^n \Delta_{\lambda_j,x_j}.$$
   Since
   $$\Delta_{\lambda,x}\left[(xy)^{-\lambda+1/2}J_{\lambda-1/2}(xy)\right]= y^2 (xy)^{-\lambda+1/2}J_{\lambda-1/2}(xy), \quad x,y \in (0,\infty)$$
   (see \cite[(5.3.7), p. 103]{L}), it follows that, for each $x=(x_1, \dots, x_n),$ $y=(y_1, \dots, y_n)\in (0,\infty)^n$,
   \begin{equation}\label{1.1}
      \Delta_{\lambda_1,\ldots,\lambda_n}\left[\prod_{j=1}^n (x_jy_j)^{-\lambda_j+\frac{1}{2}}J_{\lambda_j-\frac{1}{2}}(x_jy_j)\right]=
      \left( \sum_{j=1}^n y_j^2\right)\prod_{j=1}^n (x_jy_j)^{-\lambda_j+\frac{1}{2}}J_{\lambda_j-\frac{1}{2}}(x_jy_j).
   \end{equation}
   According to \eqref{1.1} we say that $m$ is of Laplace transform type when
   \begin{equation}\label{1.2}
      m(y)=\sum_{j=1}^n y_j^2 \int_0^\infty e^{-t \overset{n}{\underset{j=1}{\sum}} y_j^2} \phi(t)dt, \quad y\in (0,\infty)^n,
   \end{equation}
   for a certain $\phi \in L^\infty(0,\infty)$ (\cite[p. 121]{St1}).\\

   Our main result is the following.
   \begin{Th}\label{multiplier}
      Let $\lambda_j>-1/2,$ $j=1,\ldots,n.$ Assume that $m$ is of Laplace transform type. Then, the Hankel multiplier
      $T_{\lambda_1,\ldots,\lambda_n}^m$ is bounded from $L^p((0,\infty)^n, \overset{n}{\underset{j=1}{\prod}} x_j^{2\lambda_j }dx)$ into itself,
      for every $1<p<\infty,$ and from $L^1((0,\infty)^n, \overset{n}{\underset{j=1}{\prod}} x_j^{2\lambda_j }dx)$ into
      $L^{1,\infty}((0,\infty)^n, \overset{n}{\underset{j=1}{\prod}} x_j^{2\lambda_j }dx).$
   \end{Th}

   This theorem will be proved in Section 2. It can be seen as an extension to higher dimension of \cite[Theorem 1.2]{BMR}. However, in order to prove Theorem
   \ref{multiplier} we use a different procedure  than the one employed in the proof of \cite[Theorem 1.2]{BMR}. In \cite{BMR}
   Calderón-Zygmund theory for singular integral operators is applied. Here, we represent the multiplier operator $T_{\lambda_1,\ldots,\lambda_n}^m$ as a
   principal value integral operator when it acts on the space $C_c^\infty((0,\infty)^n)$ of the $C^\infty$ functions with compact support in $(0,\infty)^n$ (see Proposition 2.1). Then, after proving $L^p$-boundedness properties for the maximal operator associated with the principal
   value integral operator (see Proposition 2.2), we extend the Hankel multiplier $T_{\lambda_1,\ldots,\lambda_n}^m$ to
   $L^p((0,\infty)^n, \overset{n}{\underset{j=1}{\prod}} x_j^{2\lambda_j }dx),$ $1\leq p<\infty,$ as a principal value integral operator that is bounded from
   $L^p((0,\infty)^n, \overset{n}{\underset{j=1}{\prod}} x_j^{2\lambda_j }dx)$ into itself, when $1<p<\infty,$ and from
   $L^1((0,\infty)^n, \overset{n}{\underset{j=1}{\prod}} x_j^{2\lambda_j }dx)$ into $L^{1,\infty}((0,\infty)^n, \overset{n}{\underset{j=1}{\prod}} x_j^{2\lambda_j }dx)$ (see Proposition 2.3). To establish these properties we split the region $(0,\infty)^n\times (0,\infty)^n$ in two parts. The set
   $$ \Omega= \{ (x,y)\in (0,\infty)^n\times (0,\infty)^n: x_j/2 < y_j < 2x_j, \ j=1,\ldots,n\}$$
   is called the local region. In $\Omega$ the kernel $K_{\lambda_1,\ldots,\lambda_n}^\phi$ that defines the Hankel multiplier $T_{\lambda_1,\ldots,\lambda_n}^m$ differs
   from $\overset{n}{\underset{j=1}{\prod}}(x_jy_j)^{-\lambda_j}K^m$, where $K^m$ is the kernel associated with the Fourier multipliers $T^m$ defined by
   $$T^mf= (m\hat{f})\check{\;}, \quad f\in L^2(\mathbb{R},dx),$$
   by a kernel defining a bounded operator in
   $L^p((0,\infty)^n, \overset{n}{\underset{j=1}{\prod}} x_j^{2\lambda_j}dx),$ for every $1\leq p<\infty$.
   Here as usual, by $\hat{f}$ we denote the Fourier transform of $f$ and by $\check{g}$ the inverse Fourier transform of $g$.\\

   On $(0,\infty)^n\times (0,\infty)^n)\backslash \Omega,$ called global region, the kernel $|K_{\lambda_1,\ldots,\lambda_n}^\phi|$ defines a positive bounded operator
   from $L^p((0,\infty)^n, \overset{n}{\underset{j=1}{\prod}} x_j^{2\lambda_j}dx)$ into itself, when $1<p<\infty,$ and from
   $L^1((0,\infty)^n, \overset{n}{\underset{j=1}{\prod}} x_j^{2\lambda_j}dx)$ into $L^{1,\infty}((0,\infty)^n, \overset{n}{\underset{j=1}{\prod}} x_j^{2\lambda_j}dx)$.\\

   The semigroup $\{W_t^{\lambda_1, \dots, \lambda_n}\}_{t \geq 0}$ generated by $-\Delta_{\lambda_1, \dots, \lambda_n}$ is defined as follows
   $$W_t^{\lambda_1, \dots, \lambda_n}(f)(x) = \int_{(0,\infty)^n} \prod_{j=1}^n W_t^{\lambda_j}(x_j,y_j)f(y) \prod_{j=1}^n y_j^{2\lambda_j} dy,
     \quad x=(x_1, \dots, x_n) \in (0,\infty)^n,$$
   where
    \begin{equation}\label{F11}
    W_t^\lambda(x,y)=\frac{(xy)^{-\lambda+1/2}}{2t}I_{\lambda-1/2}\left( \frac{x y}{2t}\right) e^{-(x^2+y^2)/4t}, \quad t,x,y  \in (0,\infty), \  \lambda>-1/2,
    \end{equation}
   and $I_\nu$ denotes the modified Bessel function of the first kind and order $\nu$ (\cite[p. 395]{W}). This semigroup
   $\{W_t^{\lambda_1, \dots, \lambda_n}\}_{t \geq 0}$
   is a symmetric diffusion semigroup in the sense of Stein (\cite[p. 65]{St1}). Moreover, the Hankel multiplier $T_{\lambda_1, \dots, \lambda_n}^m$ is
   actually an spectral multiplier associated with $\Delta_{\lambda_1, \dots, \lambda_n}$. Then, by \cite[Corollary 3, p. 121]{St1}  $T_{\lambda_1, \dots, \lambda_n}^m$ is bounded from $L^p((0,\infty)^n, \overset{n}{\underset{j=1}{\prod}} x_j^{2\lambda_j}dx)$ into itself, for every $1<p<\infty$.
   In Theorem~\ref{multiplier} we prove as a new result that $T_{\lambda_1, \dots, \lambda_n}^m$ defines a bounded operator from
   $L^1((0,\infty)^n, \overset{n}{\underset{j=1}{\prod}} x_j^{2\lambda_j}dx)$ into $L^{1,\infty}((0,\infty)^n, \overset{n}{\underset{j=1}{\prod}} x_j^{2\lambda_j}dx)$.
   Moreover, in Proposition \ref{limit} we establish a representation of the operator $T_{\lambda_1, \dots, \lambda_n}^m$ as a principal value integral operator
   in $L^p((0,\infty)^n, \overset{n}{\underset{j=1}{\prod}} x_j^{2\lambda_j}dx)$, $1 \leq p < \infty$ .\\

   As an application of Theorem~\ref{multiplier} we can show $L^p$-boundedness properties for the imaginary powers of the Bessel operator $\Delta_{\lambda_1, \dots, \lambda_n}$.
   We define, for every $\beta \in \mathbb{R}$, the function
   $$\phi_\beta(t)= \frac{t^{-i \beta}}{\Gamma(1-i\beta)}, \quad t \in (0, \infty).$$
   We have that
   $$m_\beta(y) = \sum_{j=1}^n y_j^2 \int_0^\infty e^{-t \overset{n}{\underset{j=1}{\sum}} y_j^2} \phi_\beta(t) dt
                = \left( \sum_{j=1}^n y_j^2 \right)^{i\beta}, \quad y=(y_1, \dots, y_n) \in (0,\infty)^n \text{ and } \beta \in \mathbb{R}. $$
   For every $\beta \in \mathbb{R}$, the $i\beta$-power $\Delta_{\lambda_1, \dots, \lambda_n}^{i\beta}$ of $\Delta_{\lambda_1, \dots, \lambda_n}$ is defined by
   $$\Delta_{\lambda_1, \dots, \lambda_n}^{i\beta}=T_{\lambda_1, \dots, \lambda_n}^{m_\beta}.$$
   From Theorem~\ref{multiplier} we deduce the following.

   \begin{Cor}
     Let $\beta \in \mathbb{R}$. Then, the operator $\Delta_{\lambda_1, \dots, \lambda_n}^{i\beta}$ is bounded from
     $L^p((0,\infty)^n, \overset{n}{\underset{j=1}{\prod}} x_j^{2\lambda_j}dx)$ into itself, for every $1<p<\infty$, and from
     $L^1((0,\infty)^n, \overset{n}{\underset{j=1}{\prod}} x_j^{2\lambda_j}dx)$ into $L^{1,\infty}((0,\infty)^n, \overset{n}{\underset{j=1}{\prod}} x_j^{2\lambda_j}dx)$.
   \end{Cor}

   The next properties of the modified Bessel function $I_\nu$, $\nu>-1$, that can be found in \cite[Chapter 5]{L},
   are very useful in the sequel:
   \begin{equation}\label{I1}
      z^{-\nu}I_\nu (z) \sim \frac{1}{2^\nu \Gamma(\nu +1)}, \quad \text{as } z \rightarrow 0^+;
   \end{equation}
   \begin{equation}\label{I2}
      I_\nu (z) = \frac{e^z}{\sqrt{2 \pi z}}\left( \sum_{k=0}^n (-1)^k [\nu,k] (2z)^{-k} + \mathcal{O}\left(\frac{1}{z^{n+1}}\right) \right), \quad n \in \mathbb{N}, \ z \in (0,\infty),
    \end{equation}
    where $[\nu,0]=1$ and
    $$[\nu,k]=\frac{(4\nu^2-1)(4\nu^2-3^2) \ldots (4\nu^2-(2k-1)^2)}{2^{2k}\Gamma(k+1)}, \quad k=1,2, \dots \ ;$$
    and
    \begin{equation}\label{I3}
      \frac{d}{dz}\left( z^{-\nu} I_\nu(z) \right) = z^{-\nu} I_{\nu+1}(z), \quad z \in (0,\infty).
    \end{equation}

    Also, we will use the following properties of the Bessel function $J_\nu$, $\nu>-1$, (see \cite[pages 110 and 123]{L}):
    \begin{equation}\label{J1}
      z^{-\nu}J_\nu (z) \sim \frac{1}{2^\nu \Gamma(\nu +1)}, \quad \text{as } z \rightarrow 0^+;
   \end{equation}
   \begin{equation}\label{J2}
     J_\nu (z) = \sqrt{\frac{2}{\pi z}}\left( \cos\left( z- \frac{\nu \pi}{2} - \frac{\pi}{4}\right) + \mathcal{O}\left(\frac{1}{z}\right) \right), \quad z \in (0,\infty);
   \end{equation}
   and
   \begin{equation}\label{J3}
      \frac{d}{dz}\left( z^{-\nu} J_\nu(z) \right) = - z^{-\nu} J_{\nu+1}(z), \quad z \in (0,\infty).
   \end{equation}

   The heat kernel associated with the Bessel operator $-\Delta_{\lambda,x}$ can be written in terms of Bessel functions by (\cite[p. 395]{W})
   \begin{equation}\label{heatkernel}W_t^\lambda(x,y)
   = \int_0^\infty e^{-tz^2} (zx)^{-\lambda +1/2} J_{\lambda-1/2}(zx) (zy)^{-\lambda +1/2} J_{\lambda-1/2}(zy) z^{2\lambda} dz, \quad t,x,y \in (0,\infty).
   \end{equation}
   The $n$-dimensional Hankel transform on function and distribution spaces has been studied by Molina and Trione (\cite{Mo}). A
   formula connecting Hankel transform and Bessel operators
    that we will use is the following
   \begin{equation}\label{H1}
     h_{\lambda_1, \dots, \lambda_n}(\Delta_{\lambda_1, \dots, \lambda_n}f)(x)= \lvert x \rvert^2 h_{\lambda_1, \dots, \lambda_n}(f)(x),
     \quad x \in (0,\infty)^n \text{ and } f \in C_c^\infty((0,\infty)^n).
   \end{equation}
   Also the Plancherel type equality
   \begin{equation}\label{H2}
     \int_{(0,\infty)^n} h_{\lambda_1, \dots, \lambda_n}(f)(x) h_{\lambda_1, \dots, \lambda_n}(g)(x) \prod_{j=1}^n x_j^{2\lambda_j} dx
     = \int_{(0,\infty)^n} f(x)g(x) \prod_{j=1}^n x_j^{2\lambda_j} dx,
   \end{equation}
   holds for every $f$, $g \in L^2((0,\infty)^n,\overset{n}{\underset{j=1}{\prod}} x_j^{2\lambda_j}dx)$ (see \cite[Lemma 2.7]{BS}).\\

   Throughout this paper we will always denote by $C$ a suitable positive constant that can change from
   one line to the other one. Also, we will use repeatedly without
   saying it that, for every $k\in \mathbb{N}$,
   ${\underset{z>0}{\sup}}\ z^ke^{-z}<\infty$.

   \section{Proof of Theorem~\ref{multiplier}}

   In order to prove Theorem~\ref{multiplier} we need to establish a pointwise integral representation for the multiplier operator $T_{\lambda_1, \dots, \lambda_n}^m$
   as a principal value integral operator (see Proposition~\ref{integral representation} below). In the sequel we assume that
   $$m(x)=\sum_{j=1}^n x_j^2 \int_0^\infty  e^{- t \overset{n}{\underset{j=1}{\sum}} x_j^2 }\phi(t) dt, \quad x \in (0,\infty)^n,$$
   where $\phi \in L^\infty((0,\infty)^n)$.\\

   Firstly we prove the following result.

   \begin{Lem}\label{lemma2.1}
     Let $\lambda_j > -1/2$, $j=1, \dots, n$. Assume that $f \in C_c^\infty((0,\infty)^n)$. Then,
     $$T_{\lambda_1, \dots, \lambda_n}^m(f)(x)=\int_0^\infty \phi(t) h_{\lambda_1, \dots, \lambda_n}
     \left( \sum_{j=1}^n y_j^2 e^{- t \overset{n}{\underset{j=1}{\sum}} y_j^2 }  h_{\lambda_1, \dots, \lambda_n}(f)(y)\right)(x) dt, \ \text{a.e. } x \in (0,\infty)^n.$$
   \end{Lem}

   \begin{proof}

     Let $g \in C_c^\infty((0,\infty)^n)$. By using Plancherel equality \eqref{H2} we get
     \begin{align*}
       \int_{(0,\infty)^n} & T_{\lambda_1, \dots, \lambda_n}^m (f)(x)g(x) \prod_{j=1}^n x_j^{2\lambda_j} dx\\
         = &  \int_{(0,\infty)^n} \sum_{j=1}^n y_j^2 \int_0^\infty e^{-t \overset{n}{\underset{j=1}{\sum}} y_j^2} \phi(t)dt \
           h_{\lambda_1, \dots, \lambda_n}(f)(y) h_{\lambda_1, \dots, \lambda_n}(g)(y) \prod_{j=1}^n y_j^{2\lambda_j} dy\\
         = &  \int_0^\infty \phi(t) \int_{(0,\infty)^n} \sum_{j=1}^n y_j^2 e^{-t \overset{n}{\underset{j=1}{\sum}} y_j^2}  \
           h_{\lambda_1, \dots, \lambda_n}(f)(y) h_{\lambda_1, \dots, \lambda_n}(g)(y) \prod_{j=1}^n y_j^{2\lambda_j} dy dt .\\
     \end{align*}
     The interchange of the order of integration is justify because we can write
     \begin{align*}
        \int_{(0,\infty)^n} & \sum_{j=1}^n y_j^2 \int_0^\infty \lvert \phi(t) \rvert   e^{-t \overset{n}{\underset{j=1}{\sum}} y_j^2}dt \
           \lvert h_{\lambda_1, \dots, \lambda_n}(f)(y) \rvert \ \lvert h_{\lambda_1, \dots, \lambda_n}(g)(y) \rvert \prod_{j=1}^n y_j^{2\lambda_j} dy\\
         \leq & C \int_{(0,\infty)^n} \lvert h_{\lambda_1, \dots, \lambda_n}(f)(y) \rvert \ \lvert h_{\lambda_1, \dots, \lambda_n}(g)(y) \rvert \prod_{j=1}^n y_j^{2\lambda_j} dy\\
         \leq & C \lVert h_{\lambda_1, \dots, \lambda_n}(f) \rVert_{L^2((0,\infty)^n,\overset{n}{\underset{j=1}{\prod}} x_j^{2\lambda_j }dx)}
          \ \lVert h_{\lambda_1, \dots, \lambda_n}(g)\rVert_{L^2((0,\infty)^n, \overset{n}{\underset{j=1}{\prod}} x_j^{2\lambda_j }dx)}\\
         = & C \lVert f \rVert_{L^2((0,\infty)^n, \overset{n}{\underset{j=1}{\prod}} x_j^{2\lambda_j }dx)} \
           \lVert g \rVert_{L^2((0,\infty)^n, \overset{n}{\underset{j=1}{\prod}} x_j^{2\lambda_j }dx)}<\infty,
     \end{align*}
     where we have used that $h_{\lambda_1, \dots, \lambda_n}$ is an isometry in $L^2((0,\infty)^n, \overset{n}{\underset{j=1}{\prod}} x_j^{2\lambda_j }dx)$ .\\

     Then, by using again \eqref{H2}, we obtain
     \begin{align*}
       \int_{(0,\infty)^n} & T_{\lambda_1, \dots, \lambda_n}^m (f)(x)g(x) \prod_{j=1}^n x_j^{2\lambda_j} dx\\
       = &  \int_0^\infty \phi(t) \int_{(0,\infty)^n} h_{\lambda_1, \dots, \lambda_n}\left(\sum_{j=1}^n y_j^2 e^{-t \overset{n}{\underset{j=1}{\sum}} y_j^2}  \
           h_{\lambda_1, \dots, \lambda_n}(f)(y)\right)(x) g(x) \prod_{j=1}^n x_j^{2\lambda_j} dx dt .
     \end{align*}
     According to \eqref{J1} and \eqref{J2}, for each $0 \leq r \leq n$,
     \begin{align*}
       \int_0^{1/x_1} & \dots \int_0^{1/x_r} \int_{1/x_{r+1}}^\infty \dots \int_{1/x_n}^\infty
       \lvert \prod_{j=1}^n (x_jy_j)^{-\lambda_j+1/2} J_{\lambda_j-1/2}(x_jy_j) \rvert \\
       & \times \sum_{j=1}^n y_j^2 e^{-t \overset{n}{\underset{j=1}{\sum}} y_j^2}\ \lvert h_{\lambda_1, \dots, \lambda_n}(f)(y) \rvert \prod_{j=1}^n y_j^{2\lambda_j} dy\\
         \leq & C \prod_{j=r+1}^n x_j^{-\lambda_j} \int_{(0,\infty)^n} \sum_{j=1}^n y_j^2 e^{-t \overset{n}{\underset{j=1}{\sum}} y_j^2}\
           \lvert h_{\lambda_1, \dots, \lambda_n}(f)(y) \rvert \prod_{j=1}^r y_j^{2\lambda_j} \prod_{j=r+1}^n y_j^{\lambda_j}dy, \quad x \in (0,\infty)^n.
     \end{align*}
     Then, since $|y|^lh_{\lambda_1, \dots, \lambda_n}(f)$ is bounded on $(0,\infty)^n$, for every $l\in \mathbb{N}$, and $g \in C_c^\infty((0,\infty)^n)$, we get, for every $0 \leq r \leq n,$
     \begin{align*}
       \int_0^\infty & \lvert \phi(t) \rvert \int_{(0,\infty)^n} \int_0^{1/x_1}  \dots \int_0^{1/x_r} \int_{1/x_{r+1}}^\infty \dots \int_{1/x_n}^\infty
       \lvert \prod_{j=1}^n (x_jy_j)^{-\lambda_j+1/2} J_{\lambda_j-1/2}(x_jy_j) \rvert \\
       & \times \sum_{j=1}^n y_j^2 e^{-t \overset{n}{\underset{j=1}{\sum}} y_j^2}\ \lvert h_{\lambda_1, \dots, \lambda_n}(f)(y) \rvert \prod_{j=1}^n y_j^{2\lambda_j} dy
        \lvert g(x) \rvert \prod_{j=1}^n x_j^{2\lambda_j} dx dt\\
       \leq & C \int_0^\infty  \int_{(0,\infty)^n} \int_{(0,\infty)^n}
         \sum_{j=1}^n y_j^2 e^{-t \overset{n}{\underset{j=1}{\sum}} y_j^2}\ \lvert h_{\lambda_1, \dots, \lambda_n}(f)(y) \rvert
        \lvert g(x) \rvert \\
        & \times \prod_{j=1}^r y_j^{2\lambda_j}\prod_{j=r+1}^n y_j^{\lambda_j} \prod_{j=1}^r x_j^{2\lambda_j} \prod_{j=r+1}^n x_j^{\lambda_j} dy dx dt\\
        \leq & C \int_{(0,\infty)^n} \lvert g(x) \rvert \prod_{j=1}^r x_j^{2\lambda_j} \prod_{j=r+1}^n x_j^{\lambda_j} dx \\
        & \times  \int_{(0,\infty)^n} \left(\int_0^\infty    \sum_{j=1}^n y_j^2 e^{-t \overset{n}{\underset{j=1}{\sum}} y_j^2}dt\right) \ \lvert h_{\lambda_1, \dots, \lambda_n}(f)(y) \rvert
         \prod_{j=1}^r y_j^{2\lambda_j}\prod_{j=r+1}^n y_j^{\lambda_j}   dy  \\
         \leq & C \left(\int_{(0,\infty)^n} \lvert g(x) \rvert \prod_{j=1}^r x_j^{2\lambda_j} \prod_{j=r+1}^n x_j^{\lambda_j} dx\right)
           \left(\int_{(0,\infty)^n} \lvert h_{\lambda_1, \dots, \lambda_n}(f)(y) \rvert \prod_{j=1}^r y_j^{2\lambda_j}\prod_{j=r+1}^n y_j^{\lambda_j}dy\right) < \infty.
     \end{align*}

     Note that these estimates also show that the function
     $$
    G(x)= \int_0^\infty \phi(t) h_{\lambda_1, \dots, \lambda_n}
    \left( \sum_{j=1}^n y_j^2 e^{- t \overset{n}{\underset{j=1}{\sum}} y_j^2 }  h_{\lambda_1, \dots, \lambda_n}(f)(y)\right)(x) dt, \ x \in (0,\infty)^n,
    $$
    is in $L^1_{loc}((0,\infty)^n)$.

     We conclude that
     \begin{align*}
       \int_{(0,\infty)^n} & T_{\lambda_1, \dots, \lambda_n}^m (f)(x)g(x) \prod_{j=1}^n x_j^{2\lambda_j} dx\\
       = & \int_{(0,\infty)^n} \left\{\int_0^\infty \phi(t)  h_{\lambda_1, \dots, \lambda_n}\left(\sum_{j=1}^n y_j^2 e^{-t \overset{n}{\underset{j=1}{\sum}} y_j^2}  \
           h_{\lambda_1, \dots, \lambda_n}(f)(y)\right)(x) dt \right\} g(x) \prod_{j=1}^n x_j^{2\lambda_j} dx .
     \end{align*}
     Thus, the proof of this Lemma finishes because $T_{\lambda_1, \dots, \lambda_n}^m (f)\in L^2((0,\infty)^n, \overset{n}{\underset{j=1}{\prod}}x_j^{2\lambda_j}dx)\subset L^1_{loc}((0,\infty)^n)$.
   \end{proof}

   We now establish a crucial result in order to prove Theorem~\ref{multiplier}.

   \begin{Prop}\label{integral representation}
     Let $\lambda_j>-1/2$, $j=1, \dots, n$. Assume that $f \in C_c^\infty((0,\infty)^n)$. Then
     $$T_{\lambda_1, \dots, \lambda_n}^m (f)(x)= -\lim_{\varepsilon \rightarrow 0^+} \left( \alpha(\varepsilon)f(x)+\int_{(0,\infty)^n, \ \lvert y-x \rvert>\varepsilon}
       f(y) K_{\lambda_1, \dots, \lambda_n}^\phi (x,y) \prod_{j=1}^n y_j^{2\lambda_j} dy  \right), \ \text{a.e. } x \in (0,\infty)^n, $$
     where
     $$K_{\lambda_1, \dots, \lambda_n}^\phi (x,y)= \int_0^\infty \phi(t) \frac{\partial}{\partial t} \prod_{j=1}^n W_t^{\lambda_j}(x_j,y_j) dt, \quad x,y \in (0,\infty)^n,\,\,\,x\neq y,$$
     and $\alpha$ is a bounded function on $(0,\infty)$. Moreover, if there exists the limit $\displaystyle \phi(0^+)=\lim_{t \rightarrow 0^+} \phi(t)$, then
     $$T_{\lambda_1, \dots, \lambda_n}^m (f)(x)= C \phi(0^+) f(x)-\lim_{\varepsilon \rightarrow 0^+} \left(\int_{(0,\infty)^n, \ \lvert y-x \rvert>\varepsilon}
       f(y) K_{\lambda_1, \dots, \lambda_n}^\phi (x,y) \prod_{j=1}^n y_j^{2\lambda_j} dy \right)  , \ \text{a.e. } x \in (0,\infty)^n,$$
     being $C$ a positive constant.
   \end{Prop}

   \begin{proof}
     Assume that $n\ge 2$. When $n=1$ we can proceed in a similar way. By Lemma~\ref{lemma2.1} and \eqref{H1} we can write, for a.e. $x \in (0,\infty)^n,$
     \begin{equation}\label{2.2}
       T_{\lambda_1, \dots, \lambda_n}^m(f)(x)=\int_0^\infty \phi(t) h_{\lambda_1, \dots, \lambda_n}
       \left( e^{- t \overset{n}{\underset{j=1}{\sum}} y_j^2 }  h_{\lambda_1, \dots, \lambda_n}(\Delta_{\lambda_1, \dots, \lambda_n}f)(y)\right)(x) dt.
     \end{equation}
     Fix $x\in (0,\infty)^n$ such that (\ref{2.2}) holds. According to \eqref{J1} and \eqref{J2}, for each $0 \leq r,s \leq n$, it follows that,
     \begin{align*}
       \int_0^{1/x_1} & \dots \int_0^{1/x_r} \int_{1/x_{r+1}}^\infty \dots \int_{1/x_n}^\infty \prod_{j=1}^n \left\lvert (x_jy_j)^{-\lambda_j+1/2}J_{\lambda_j-1/2}(x_jy_j) \right\rvert
        e^{- t \overset{n}{\underset{j=1}{\sum}} y_j^2 } \\
       & \times \int_0^{1/y_1}  \dots \int_0^{1/y_s} \int_{1/y_{s+1}}^\infty \dots \int_{1/y_n}^\infty \prod_{j=1}^n \left\lvert (y_jz_j)^{-\lambda_j+1/2}J_{\lambda_j-1/2}(y_jz_j) \right\rvert \\
       & \times  \left \lvert \Delta_{\lambda_1, \dots, \lambda_n}f(z) \right \rvert \prod_{j=1}^n z_j^{2\lambda_j} dz \prod_{j=1}^n y_j^{2\lambda_j} dy\\
       \leq & C \prod_{j=r+1}^n x_j^{-\lambda_j} \int_{(0,\infty)^n}  e^{- t \overset{n}{\underset{j=1}{\sum}} y_j^2 } \prod_{j=r+1}^n y_j^{-\lambda_j}\prod_{j=1}^s y_j^{2\lambda_j} \prod_{j=s+1}^n y_j^{\lambda_j}dy \\
       & \times  \int_{(0,\infty)^n}   \left \lvert \Delta_{\lambda_1, \dots, \lambda_n}f(z) \right \rvert
                \prod_{j=1}^s z_j^{2\lambda_j} \prod_{j=s+1}^n z_j^{\lambda_j} dz <\infty, \quad t>0.\\
     \end{align*}
     Then, by interchanging the order of integration and by using \eqref{heatkernel}, we get
     \begin{align}\label{2.3}
       h_{\lambda_1, \dots, \lambda_n}& \left( e^{- t \overset{n}{\underset{j=1}{\sum}} y_j^2 }  h_{\lambda_1, \dots, \lambda_n}(\Delta_{\lambda_1, \dots, \lambda_n}f)(y)\right)(x) \nonumber\\
       =& \int_{(0,\infty)^n} \Delta_{\lambda_1, \dots, \lambda_n}f(z) \prod_{j=1}^n  \int_0^\infty e^{- t y_j^2 } (x_jy_j)^{-\lambda_j+1/2}J_{\lambda_j-1/2}(x_jy_j) \nonumber\\
        & \times (y_jz_j)^{-\lambda_j+1/2}J_{\lambda_j-1/2}(y_jz_j) y_j^{2\lambda_j} dy_j  \prod_{j=1}^n z_j^{2\lambda_j} dz \nonumber\\
        = & \int_{(0,\infty)^n} \Delta_{\lambda_1, \dots, \lambda_n}f(z) \prod_{j=1}^n W_t^{\lambda_j} (x_j,z_j) \prod_{j=1}^n z_j^{2\lambda_j} dz,
        \quad t>0 ,
     \end{align}
     where, for every $\alpha>-1/2$, $W_t^\alpha(u,v)$, $t,u,v\in (0,\infty)$, is defined in (\ref{F11}).

    We choose $a>1$ such that $supp\,f\subset K^n$, where $K=[1/a,a]$. If $\alpha>-1/2$, by \eqref{I1}, \eqref{I2} and \eqref{I3}, we obtain, for every $b>0$,
     \begin{align*}
       \Big \lvert W_t^\alpha(u,v) - &\frac{t^{-\alpha-1/2}}{2^{2\alpha}\Gamma(\alpha+1/2)} \Big \rvert
         = \Big \lvert \frac{1}{(2t)^{\alpha+1/2}} \left( \frac{uv}{2t}\right)^{-\alpha+1/2} I_{\alpha-1/2}\left(\frac{uv}{2t}\right)e^{-\frac{u^2+v^2}{4t}} - \frac{t^{-\alpha-1/2}}{2^{2\alpha}\Gamma(\alpha+1/2)} \Big  \rvert\\
         \leq & \frac{1}{(2t)^{\alpha+1/2}} e^{-\frac{u^2+v^2}{4t}} \Big \lvert \left( \frac{uv}{2t}\right)^{-\alpha+1/2} I_{\alpha-1/2}\left(\frac{uv}{2t}\right) - \frac{1}{2^{\alpha-1/2}\Gamma(\alpha+1/2)} \Big  \rvert\\
       & + \frac{t^{-\alpha-1/2}}{2^{2\alpha}\Gamma(\alpha+1/2)} \Big \lvert e^{-\frac{u^2+v^2}{4t}} -1 \Big \rvert\\
       \leq & \frac{uv}{(2t)^{\alpha+3/2}} e^{-\frac{u^2+v^2}{4t}}  \sup_{z \in (0,\frac{uv}{2t})} \Big \lvert \frac{d}{dz} \left( z^{-\alpha+1/2} I_{\alpha-1/2}(z) \right) \Big \rvert
         + C \frac{u^2+v^2}{t^{\alpha+3/2}}\\
         \leq & C \left( \frac{(uv)^2}{t^{\alpha+5/2}}  + \frac{u^2+v^2}{t^{\alpha+3/2}} \right)\\
         \leq & C \frac{1}{t^{\alpha+3/2}} , \quad t>1,\,\,0<u\le b,\,\and\,\,v\in K.
     \end{align*}
     Then, by using again (\ref{I1}) and (\ref{I2}), it follows that
     \begin{align}\label{2.4}
       \Big \lvert \prod_{j=1}^n W_t^{\lambda_j}(x_j,y_j)&
       - \frac{t^{-\overset{n}{\underset{j=1}{\sum}}(\lambda_j+1/2)}}{\overset{n}{\underset{j=1}{\prod}}2^{2\lambda_j}\Gamma(\lambda_j+1/2)} \Big \rvert \nonumber\\
       \leq & \sum_{i=1}^n \prod_{j=1}^{i-1} \frac{t^{-(\lambda_j+1/2)}}{2^{2\lambda_j}\Gamma(\lambda_j+1/2)}
       \Big \lvert W_t^{\lambda_i}(x_i,y_i)
       - \frac{t^{-(\lambda_i+1/2)}}{2^{2\lambda_i}\Gamma(\lambda_i+1/2)} \Big \rvert \prod_{j=i+1}^n W_t^{\lambda_j}(x_j,y_j) \nonumber\\
       \leq &  C \frac{1}{t^{\overset{n}{\underset{j=1}{\sum}}(\lambda_j+1/2)+1}},
        \quad t>1 \text{ and }  y \in K^n.
     \end{align}

     On the other hand, by \eqref{J1} and \eqref{H1} we have
     \begin{align}\label{2.5}
       \int_{(0,\infty)^n}& \Delta_{\lambda_1, \dots, \lambda_n}f(z) \prod_{j=1}^n z_j^{2\lambda_j} dz  =  \prod_{j=1}^n 2^{\lambda_j-1/2}\Gamma(\lambda_j+1/2) \nonumber \\
       \times &\lim_{y \rightarrow 0} \int_{(0,\infty)^n} \Delta_{\lambda_1, \dots, \lambda_n}f(z) \prod_{j=1}^n (z_jy_j)^{-\lambda_j+1/2}
         J_{\lambda_j-1/2}(z_jy_j) \prod_{j=1}^n z_j^{2\lambda_j} dz \nonumber \\
        = &  \prod_{j=1}^n 2^{\lambda_j-1/2}\Gamma(\lambda_j+1/2) \lim_{y \rightarrow 0} \sum_{j=1}^n y_j^2 h_{\lambda_1, \dots, \lambda_n}(f)(y)=0.
     \end{align}
     According to \eqref{2.2}, \eqref{2.3}, and \eqref{2.5} (suggested by \eqref{2.4}), we can write
     \begin{align}\label{2.6}
       T_{\lambda_1, \dots, \lambda_n}^m(f)(x)=&\int_0^\infty \phi(t) \int_{(0,\infty)^n} \Delta_{\lambda_1, \dots, \lambda_n}f(z)
       \left(\prod_{j=1}^n W_t^{\lambda_j} (x_j,z_j) - \frac{\chi_{(1,\infty)}(t)t^{-\overset{n}{\underset{j=1}{\sum}}(\lambda_j+1/2)}}
       {\overset{n}{\underset{j=1}{\prod}}2^{2\lambda_j}\Gamma(\lambda_j+1/2)} \right) \nonumber\\
       & \times \prod_{j=1}^n z_j^{2\lambda_j} dz dt.
     \end{align}
     The last integral is absolutely convergent. Indeed, in order to do this we split the integral
     in the following way
     \begin{align*}
       \int_0^\infty& \lvert \phi(t) \rvert \int_{(0,\infty)^n} \left \lvert \Delta_{\lambda_1, \dots, \lambda_n}f(z) \right\rvert \
       \left \lvert \prod_{j=1}^n W_t^{\lambda_j} (x_j,z_j) - \frac{\chi_{(1,\infty)}(t)t^{-\overset{n}{\underset{j=1}{\sum}}(\lambda_j+1/2)}}
       {\overset{n}{\underset{j=1}{\prod}}2^{2\lambda_j}\Gamma(\lambda_j+1/2)} \right \rvert \prod_{j=1}^n z_j^{2\lambda_j} dz dt \\
       \leq & C \left( \int_0^1 + \int_1^\infty \right) \int_{(0,\infty)^n}\left \lvert \Delta_{\lambda_1, \dots, \lambda_n}f(z) \right\rvert
       \left \lvert \prod_{j=1}^n W_t^{\lambda_j} (x_j,z_j) - \frac{\chi_{(1,\infty)}(t)t^{-\overset{n}{\underset{j=1}{\sum}}(\lambda_j+1/2)}}
       {\overset{n}{\underset{j=1}{\prod}}2^{2\lambda_j}\Gamma(\lambda_j+1/2)} \right \rvert
        \prod_{j=1}^n z_j^{2\lambda_j} dz dt\\
       = & I_1(x) + I_2(x).
     \end{align*}
     Since by using the inversion formula for Hankel transform and \eqref{heatkernel} we get
     $$\int_0^\infty W_t^\alpha(u,z)z^{2\alpha} dz=1, \quad u \in (0,\infty),$$
     when $\alpha>-1/2$, it follows that
     $$ I_1(x) \leq C \int_0^1 \int_{(0,\infty)^n} \left \lvert \Delta_{\lambda_1, \dots, \lambda_n}f(z) \right\rvert
     \prod_{j=1}^n W_t^{\lambda_j} (x_j,z_j)\prod_{j=1}^n z_j^{2\lambda_j} dz dt \leq C. $$
     Also, by \eqref{2.4} we obtain
     \begin{align*}
       I_2(x) \leq & C\int_1^\infty \frac{1}{t^{\overset{n}{\underset{j=1}{\sum}}(\lambda_j+1/2)+1}}
          \int_{K^n}  \left \lvert \Delta_{\lambda_1, \dots, \lambda_n}f(z) \right\rvert
          \prod_{j=1}^n z_j^{2\lambda_j} dz dt\le C.
     \end{align*}

     By \eqref{2.6} we have
     \begin{align}\label{2.7}
       T_{\lambda_1, \dots, \lambda_n}^m(f)(x)=& \lim_{\varepsilon \rightarrow 0^+}\int_0^\infty \phi(t)
       \int_{\tiny{\begin{array}{l}(0,\infty)^n\\ \lvert z-x \rvert > \varepsilon \end{array}}} \Delta_{\lambda_1, \dots, \lambda_n}f(z) \nonumber \\
       & \times \left(\prod_{j=1}^n W_t^{\lambda_j} (x_j,z_j) - \frac{\chi_{(1,\infty)}(t)t^{-\overset{n}{\underset{j=1}{\sum}}(\lambda_j+1/2)}}
       {\overset{n}{\underset{j=1}{\prod}}2^{2\lambda_j}\Gamma(\lambda_j+1/2)} \right) \prod_{j=1}^n z_j^{2\lambda_j} dz dt.
     \end{align}

     Assume that $\varepsilon$ is small enough, for instance, $0<\varepsilon<x_i/2$, $i=1,...,n$. We now analyze the integral
     $$I^\varepsilon(x,t)= \int_{\tiny{\begin{array}{l}K^n\\ \lvert z-x \rvert > \varepsilon \end{array}}}
     \Delta_{\lambda_1,z_1}f(z)
     \left(\prod_{j=1}^n W_t^{\lambda_j} (x_j,z_j) - \frac{\chi_{(1,\infty)}(t)t^{-\overset{n}{\underset{j=1}{\sum}}(\lambda_j+1/2)}}
     {\overset{n}{\underset{j=1}{\prod}}2^{2\lambda_j}\Gamma(\lambda_j+1/2)} \right) \prod_{j=1}^n z_j^{2\lambda_j} dz,\,\,\,t>0.$$
     The study of the integral involving $\Delta_{\lambda_j,z_j}$, $j=2, \dots, n$, can be made in a similar way. We write $\bar{y}=(y_2, \dots, y_n)$
     when $y=(y_1, y_2, \dots, y_n) \in (0,\infty)^n$. We can write
     \begin{align}\label{2.8}
       I^\varepsilon(x,t) & = \int_{\lvert \bar{z} - \bar{x} \rvert > \varepsilon} \int_0^\infty \Delta_{\lambda_1,z_1}f(z)
       \left(\prod_{j=1}^n W_t^{\lambda_j} (x_j,z_j) - \frac{\chi_{(1,\infty)}(t)t^{-\overset{n}{\underset{j=1}{\sum}}(\lambda_j+1/2)}}
       {\overset{n}{\underset{j=1}{\prod}}2^{2\lambda_j}\Gamma(\lambda_j+1/2)} \right) \prod_{j=1}^n z_j^{2\lambda_j} dz \nonumber\\
       & + \int_{\lvert \bar{z} - \bar{x} \rvert < \varepsilon} \left( \int_0^{x_1-\sqrt{\varepsilon^2 - \lvert \bar{z} - \bar{x} \rvert^2}}
         + \int_{x_1+\sqrt{\varepsilon^2 - \lvert \bar{z} - \bar{x} \rvert^2}}^\infty \right) \Delta_{\lambda_1,z_1}f(z) \nonumber\\
       & \times \left(\prod_{j=1}^n W_t^{\lambda_j} (x_j,z_j) - \frac{\chi_{(1,\infty)}(t)t^{-\overset{n}{\underset{j=1}{\sum}}(\lambda_j+1/2)}}
       {\overset{n}{\underset{j=1}{\prod}}2^{2\lambda_j}\Gamma(\lambda_j+1/2)} \right) \prod_{j=1}^n z_j^{2\lambda_j} dz \nonumber\\
       = & I_1^\varepsilon(x,t) + I_2^\varepsilon(x,t), \quad t>0.
     \end{align}
     By partial integration we obtain
     \begin{align*}
       \int_0^\infty & \Delta_{\lambda_1,z_1}f(z)
       \left(\prod_{j=1}^n W_t^{\lambda_j} (x_j,z_j) - \frac{\chi_{(1,\infty)}(t)t^{-\overset{n}{\underset{j=1}{\sum}}(\lambda_j+1/2)}}
       {\overset{n}{\underset{j=1}{\prod}}2^{2\lambda_j}\Gamma(\lambda_j+1/2)} \right) z_1^{2\lambda_1} dz_1\\
       =& - z_1^{2\lambda_1} \frac{\partial}{\partial z_1}f(z)
        \left.\left(\prod_{j=1}^n W_t^{\lambda_j} (x_j,z_j) - \frac{\chi_{(1,\infty)}(t)t^{-\overset{n}{\underset{j=1}{\sum}}(\lambda_j+1/2)}}
       {\overset{n}{\underset{j=1}{\prod}}2^{2\lambda_j}\Gamma(\lambda_j+1/2)} \right)\right]_0^\infty\\
       & + \left. z_1^{2\lambda_1} \frac{\partial}{\partial z_1}\left(\prod_{j=1}^n W_t^{\lambda_j} (x_j,z_j)\right) f(z)\right]_0^\infty
         + \int_0^\infty f(z) \Delta_{\lambda_1, z_1} W_t^{\lambda_1}(x_1,z_1)\prod_{j=2}^n W_t^{\lambda_j} (x_j,z_j) z_1^{2\lambda_1} dz_1\\
         =& \int_ {1/a}^a f(z) \Delta_{\lambda_1, z_1} W_t^{\lambda_1}(x_1,z_1)\prod_{j=2}^n W_t^{\lambda_j} (x_j,z_j) z_1^{2\lambda_1} dz_1,
         \quad t>0 \text{ and } \bar{z} \in K^{n-1}.
     \end{align*}
      Differentiating in (\ref{heatkernel}), by using \eqref{J3} and \cite[(5.3.6), p. 103]{L}, we get  $-\Delta_{\alpha,u}W_t^\alpha(u,v)=\frac{\partial}{\partial t}W_t^\alpha(u,v)$, $u,v,t>0$, $\alpha>-1/2$, and it follows that
     \begin{align}\label{2.9}
       I_1^\varepsilon(x,t) & =
       \int_{\lvert \bar{z} - \bar{x} \rvert > \varepsilon} \int_0^\infty \Delta_{\lambda_1,z_1}f(z)
       \left(\prod_{j=1}^n W_t^{\lambda_j} (x_j,z_j) - \frac{\chi_{(1,\infty)}(t)t^{-\overset{n}{\underset{j=1}{\sum}}(\lambda_j+1/2)}}
       {\overset{n}{\underset{j=1}{\prod}}2^{2\lambda_j}\Gamma(\lambda_j+1/2)} \right) \prod_{j=1}^n z_j^{2\lambda_j} dz \nonumber\\
       & = -\int_{\lvert \bar{z} - \bar{x} \rvert > \varepsilon} \int_0^\infty f(z)
       \frac{\partial}{\partial t} W_t^{\lambda_1}(x_1,z_1)\prod_{j=2}^n W_t^{\lambda_j} (x_j,z_j) \prod_{j=1}^n z_j^{2\lambda_j} dz, \quad t>0.
     \end{align}

     In a similar way we obtain
     \begin{align*}
       \Big(&\int_0^{x_1-\sqrt{\varepsilon^2 - \lvert \bar{z} - \bar{x} \rvert^2}}
         + \int_{x_1+\sqrt{\varepsilon^2 - \lvert \bar{z} - \bar{x} \rvert^2}}^\infty \Big) \Delta_{\lambda_1,z_1}f(z)
       \left(\prod_{j=1}^n W_t^{\lambda_j} (x_j,z_j) - \frac{\chi_{(1,\infty)}(t)t^{-\overset{n}{\underset{j=1}{\sum}}(\lambda_j+1/2)}}
       {\overset{n}{\underset{j=1}{\prod}}2^{2\lambda_j}\Gamma(\lambda_j+1/2)} \right)z_1^{2\lambda_1}dz_1 \\
       =& \Big(\int_0^{x_1-\sqrt{\varepsilon^2 - \lvert \bar{z} - \bar{x} \rvert^2}}
         + \int_{x_1+\sqrt{\varepsilon^2 - \lvert \bar{z} - \bar{x} \rvert^2}}^\infty \Big)
         f(z) \Delta_{\lambda_1, z_1} W_t^{\lambda_1}(x_1,z_1)\prod_{j=2}^n W_t^{\lambda_j} (x_j,z_j) z_1^{2\lambda_1} dz_1\\
       & -\left( H_1(x,x_1-\sqrt{\varepsilon^2 - \lvert \bar{z} - \bar{x} \rvert^2},\bar{z},t)
         - H_1(x,x_1+\sqrt{\varepsilon^2 - \lvert \bar{z} - \bar{x} \rvert^2},\bar{z},t)\right)\\
       & + H_2(x,x_1-\sqrt{\varepsilon^2 - \lvert \bar{z} - \bar{x} \rvert^2},\bar{z},t)
         - H_2(x,x_1+\sqrt{\varepsilon^2 - \lvert \bar{z} - \bar{x} \rvert^2},\bar{z},t),\,\,|\bar{x}-\bar{z}|<\varepsilon\,\,\,and\,\,\,t>0.
     \end{align*}
     where
     $$H_1(x,z,t)= z_1^{2\lambda_1} \frac{\partial}{\partial z_1}f(z)
        \left(\prod_{j=1}^n W_t^{\lambda_j} (x_j,z_j) - \frac{\chi_{(1,\infty)}(t)t^{-\overset{n}{\underset{j=1}{\sum}}(\lambda_j+1/2)}}
       {\overset{n}{\underset{j=1}{\prod}}2^{2\lambda_j}\Gamma(\lambda_j+1/2)} \right), \ t>0 \text{ and } z \in (0,\infty)^n,$$
     and
     $$H_2(x,z,t)= z_1^{2\lambda_1} \frac{\partial}{\partial z_1}\left(\prod_{j=1}^n W_t^{\lambda_j} (x_j,z_j)\right) f(z), \quad t>0 \text{ and } z \in (0,\infty)^n.$$
     We have, by \eqref{2.9}, that
     \begin{align}\label{2.10}
       I_1^\varepsilon (x,t)& + \int_{\lvert \bar{z} - \bar{x} \rvert < \varepsilon} \left( \int_0^{x_1-\sqrt{\varepsilon^2 - \lvert \bar{z} - \bar{x} \rvert^2}}
         + \int_{x_1+\sqrt{\varepsilon^2 - \lvert \bar{z} - \bar{x} \rvert^2}}^\infty \right)
         f(z) \Delta_{\lambda_1, z_1} W_t^{\lambda_1}(x_1,z_1)\prod_{j=2}^n W_t^{\lambda_j} (x_j,z_j) \prod_{j=1}^n z_j^{2\lambda_j} dz \nonumber\\
       & = -\int_{\lvert z - x \rvert > \varepsilon} f(z)
       \frac{\partial}{\partial t} W_t^{\lambda_1}(x_1,z_1)\prod_{j=2}^n W_t^{\lambda_j} (x_j,z_j) \prod_{j=1}^n z_j^{2\lambda_j} dz, \quad t>0.
     \end{align}

      Moreover, according to
     \eqref{I1}, \eqref{I2}, and \eqref{I3} it follows that, if $\mathbb{K}$ is a compact subset of $(0,\infty)$, for every $z_1\in \mathbb{K}$,
     \begin{align}\label{2.11}
       \Big \lvert  \frac{\partial}{\partial z_1}  W_t^{\lambda_1}(x_1,z_1) \Big\rvert
         =& \Big \lvert \frac{\partial}{\partial z_1} \left( \frac{1}{(2t)^{\lambda_1+1/2}} \left( \frac{x_1z_1}{2t} \right)^{-\lambda_1+1/2}
         I_{\lambda_1-1/2}\left(  \frac{x_1z_1}{2t} \right) e^{-(x_1^2+z_1^2)/4t} \right) \Big\rvert \nonumber\\
         = & \frac{1}{(2t)^{\lambda_1+1/2}}
             \Big \lvert \frac{x_1}{2t} \left( \frac{x_1z_1}{2t} \right)^{-\lambda_1+1/2} I_{\lambda_1+1/2}\left(  \frac{x_1z_1}{2t} \right) \nonumber\\
            & - \frac{z_1}{2t}\left( \frac{x_1z_1}{2t} \right)^{-\lambda_1+1/2} I_{\lambda_1-1/2}\left(  \frac{x_1z_1}{2t} \right)\Big \rvert e^{-(x_1^2+z_1^2)/4t}\nonumber\\
         \leq & C \left\{\begin{array}{lcl}
                  \dfrac{1}{t^{\lambda_1+3/2}} &,& \,\,t\ge 1\\
                  &&\\
                  \dfrac{1}{t}e^{-(x_1-z_1)^2/8t}&,& \,\,0<t<1.\\
                \end{array} \right.
     \end{align}
     Then, since $f\in C_c^\infty((0,\infty)^n)$, by using main value theorem, \eqref{I1}, \eqref{2.4} and
     \eqref{2.11} imply that
     \begin{align}\label{2.12}
       \big \lvert H_1(x,x_1-&\sqrt{\varepsilon^2 - \lvert \bar{z} - \bar{x} \rvert^2},\bar{z},t)
        - H_1(x,x_1+\sqrt{\varepsilon^2 - \lvert \bar{z} - \bar{x} \rvert^2},\bar{z},t) \big \rvert \nonumber\\
         \leq & \big \lvert ( x_1-\sqrt{\varepsilon^2 - \lvert \bar{z} - \bar{x} \rvert^2})^{2\lambda_1} \Big(\frac{\partial}{\partial z_1}
         f\Big)(x_1-\sqrt{\varepsilon^2 - \lvert \bar{z} - \bar{x} \rvert^2},\bar{z}) \nonumber\\
         & - ( x_1+\sqrt{\varepsilon^2 - \lvert \bar{z} - \bar{x} \rvert^2})^{2\lambda_1} \Big(\frac{\partial}{\partial z_1}
         f\Big)(x_1+\sqrt{\varepsilon^2 - \lvert \bar{z} - \bar{x} \rvert^2},\bar{z}) \big\rvert \nonumber\\
         & \times \Big \lvert W_t^{\lambda_1}(x_1,x_1+\sqrt{\varepsilon^2 - \lvert \bar{z} - \bar{x} \rvert^2})\prod_{j=2}^n W_t^{\lambda_j} (x_j,z_j)
         - \frac{\chi_{(1,\infty)}(t)t^{-\overset{n}{\underset{j=1}{\sum}}(\lambda_j+1/2)}}
         {\overset{n}{\underset{j=1}{\prod}}2^{2\lambda_j}\Gamma(\lambda_j+1/2)} \Big \rvert  \nonumber\\
         & + \Big \lvert ( x_1-\sqrt{\varepsilon^2 - \lvert \bar{z} - \bar{x} \rvert^2})^{2\lambda_1} \left(\frac{\partial}{\partial z_1}
         f\right)(x_1-\sqrt{\varepsilon^2 - \lvert \bar{z} - \bar{x} \rvert^2},\bar{z}) \Big \rvert \nonumber \\
         & \times \Big \lvert W_t^{\lambda_1}(x_1,x_1-\sqrt{\varepsilon^2 - \lvert \bar{z} - \bar{x} \rvert^2})
         - W_t^{\lambda_1}(x_1,x_1+\sqrt{\varepsilon^2 - \lvert \bar{z} - \bar{x} \rvert^2}) \Big \rvert \prod_{j=2}^n W_t^{\lambda_j}(x_j,z_j) \nonumber\\
         \leq & C \frac{\varepsilon}{t^{\overset{n}{\underset{j=1}{\sum}}(\lambda_j+1/2)+1}},\,\, t\geq 1,\,\,  \bar{z} \in K^{n-1} \,\,\text{ and }\,\, \lvert \bar{z}-\bar{x} \rvert<\varepsilon.
     \end{align}


     Also, from \eqref{I2} and \eqref{2.11} we deduce
     \begin{align}\label{2.14}
       \big \lvert H_1(x,x_1-&\sqrt{\varepsilon^2 - \lvert \bar{z} - \bar{x} \rvert^2},\bar{z},t)
        - H_1(x,x_1+\sqrt{\varepsilon^2 - \lvert \bar{z} - \bar{x} \rvert^2},\bar{z},t) \big \rvert \nonumber\\
        \leq & C \left\{ \sqrt{\varepsilon^2 - \lvert \bar{z} - \bar{x} \rvert^2}
           \frac{(x_1(x_1+\sqrt{\varepsilon^2 - \lvert \bar{z} - \bar{x} \rvert^2}))^{-\lambda_1}}{\sqrt{t}} \right.\nonumber\\
        & + \Big \lvert ( x_1-\sqrt{\varepsilon^2 - \lvert \bar{z} - \bar{x} \rvert^2})^{2\lambda_1} \Big(\frac{\partial}{\partial z_1}
         f\Big)(x_1-\sqrt{\varepsilon^2 - \lvert \bar{z} - \bar{x} \rvert^2},\bar{z}) \Big \rvert  \nonumber\\
        & \times \left. \int_{x_1-\sqrt{\varepsilon^2 - \lvert \bar{z} - \bar{x} \rvert^2}}^{x_1+\sqrt{\varepsilon^2 - \lvert \bar{z} - \bar{x} \rvert^2}}
           \Big\lvert \frac{\partial}{\partial z_1} W_t^{\lambda_1}(x_1,z_1)\Big\rvert dz_1  \right\} \prod_{j=2}^n W_t^{\lambda_j}(x_j,z_j) \nonumber \\
         \leq & C \left\{  \frac{\varepsilon}{\sqrt{t}}
          +
           \int_{x_1-\sqrt{\varepsilon^2 - \lvert \bar{z} - \bar{x} \rvert^2}}^{x_1+\sqrt{\varepsilon^2 - \lvert \bar{z} - \bar{x} \rvert^2}}
           \frac{e^{- \lvert x_1-z_1 \rvert^2/8t}}{t} dz_1 \right\} \nonumber\\
         & \times  \prod_{j=2}^n \frac{1}{\sqrt{t}}e^{-(x_j-z_j)^2/4t}  ,
         \ 0<t<1, \,\, \bar{z} \in K^{n-1} \text{ and } \lvert \bar{z}-\bar{x} \rvert<\varepsilon.
     \end{align}

     By \eqref{2.12}, we can write
     \begin{align}\label{2.15.a}
       \int_1^\infty & \lvert \phi(t) \rvert \int_{\bar{z}\in K^{n-1},\,\,\lvert \bar{z} - \bar{x} \rvert < \varepsilon }
        \big \lvert H_1(x,x_1-\sqrt{\varepsilon^2 - \lvert \bar{z} - \bar{x} \rvert^2},\bar{z},t)
        - H_1(x,x_1+\sqrt{\varepsilon^2 - \lvert \bar{z} - \bar{x} \rvert^2},\bar{z},t) \big \rvert \prod_{j=2}^n z_j^{2\lambda_j} d\bar{z} dt \nonumber\\
        \leq & C \varepsilon \int_1^\infty \int_{\lvert \bar{z} - \bar{x} \rvert < \varepsilon }
        \frac{1}{t^{\overset{n}{\underset{j=1}{\sum}}(\lambda_j+1/2)+1}} \prod_{j=2}^n z_j^{2\lambda_j} d\bar{z} dt
           \longrightarrow 0 , \ \varepsilon \rightarrow 0^+.
     \end{align}

     By \eqref{2.14} and by using \cite[Lemma 1.1]{ST}, when $n>2$ it follows
     \begin{align}\label{2.15.b}
       \int_0^1 & \lvert \phi(t) \rvert \int_{\bar{z}\in K^{n-1},\,\,\lvert \bar{z} - \bar{x} \rvert < \varepsilon }
        \big \lvert H_1(x,x_1-\sqrt{\varepsilon^2 - \lvert \bar{z} - \bar{x} \rvert^2},\bar{z},t)
        - H_1(x,x_1+\sqrt{\varepsilon^2 - \lvert \bar{z} - \bar{x} \rvert^2},\bar{z},t) \big \rvert \prod_{j=2}^n z_j^{2\lambda_j} d\bar{z} dt\nonumber\\
        \leq & C\Big(\varepsilon \int_0^1 \int_{\bar{z}\in K^{n-1},\,\,\lvert \bar{z} - \bar{x} \rvert < \varepsilon }\frac{e^{-\lvert \bar{z} - \bar{x} \rvert ^2/4t}}{t^{n/2}}d\bar{z}dt
        + \int_0^1   \int_{\bar{z}\in K^{n-1},\,\,\lvert \bar{z} - \bar{x} \rvert < \varepsilon }
        \int_{x_1-\sqrt{\varepsilon^2 - \lvert \bar{z} - \bar{x} \rvert^2}}^{x_1+\sqrt{\varepsilon^2 - \lvert \bar{z} - \bar{x} \rvert^2}}\frac{e^{-\lvert z - x \rvert ^2/8t}}{t^{(n+1)/2}}dzdt\Big)\nonumber\\
        \leq& C\Big(\varepsilon \int_{\lvert \bar{z} - \bar{x} \rvert<\varepsilon}\frac{d\bar{z}}{\lvert \bar{z} - \bar{x} \rvert ^{n-2}}+\int_{\lvert {z} - {x} \rvert<\varepsilon}\frac{d{z}}{\lvert {z} - {x} \rvert ^{n-1}}\Big)\to 0,\,\,\mbox{as}\,\,\varepsilon\to 0^+.
     \end{align}
     For $n=2$ we can proceed analogously.

     We now write, for each $t>0$, $\bar{z} \in K^{n-1}$ and $\lvert \bar{z} - \bar{x} \rvert < \varepsilon$,
     \begin{align*}
       \big \lvert H_2&(x,x_1-\sqrt{\varepsilon^2 - \lvert \bar{z} - \bar{x} \rvert^2},\bar{z},t)
        - H_2(x,x_1+\sqrt{\varepsilon^2 - \lvert \bar{z} - \bar{x} \rvert^2},\bar{z},t) \big \rvert \\
        \leq & \big \lvert (x_1-\sqrt{\varepsilon^2 - \lvert \bar{z} - \bar{x} \rvert^2})^{2\lambda_1} f(x_1-\sqrt{\varepsilon^2 - \lvert \bar{z} - \bar{x} \rvert^2},\bar{z})
        - (x_1+\sqrt{\varepsilon^2 - \lvert \bar{z} - \bar{x} \rvert^2})^{2\lambda_1}f(x_1+\sqrt{\varepsilon^2 - \lvert \bar{z} - \bar{x} \rvert^2},\bar{z}) \big \rvert \\
        & \times \Big \lvert \frac{\partial}{\partial z_1} W_t^{\lambda_1}(x_1,x_1+\sqrt{\varepsilon^2 - \lvert \bar{z} - \bar{x} \rvert^2}) \Big \rvert \prod_{j=2}^n W_t^{\lambda_j}(x_j,z_j)\\
        & + \big \lvert (x_1-\sqrt{\varepsilon^2 - \lvert \bar{z} - \bar{x} \rvert^2})^{2\lambda_1} f(x_1-\sqrt{\varepsilon^2 - \lvert \bar{z} - \bar{x} \rvert^2},\bar{z}) \big \rvert \\
        & \times \Big \lvert \frac{\partial}{\partial z_1} W_t^{\lambda_1}(x_1,x_1-\sqrt{\varepsilon^2 - \lvert \bar{z} - \bar{x} \rvert^2})
        - \frac{\partial}{\partial z_1} W_t^{\lambda_1}(x_1,x_1+\sqrt{\varepsilon^2 - \lvert \bar{z} - \bar{x} \rvert^2})\Big \rvert \prod_{j=2}^n W_t^{\lambda_j}(x_j,z_j).
     \end{align*}

     By  \eqref{2.11}, mean value theorem leads to
     \begin{align}\label{2.16}
        \big \lvert (x_1-&\sqrt{\varepsilon^2 - \lvert \bar{z} - \bar{x} \rvert^2})^{2\lambda_1} f(x_1-\sqrt{\varepsilon^2 - \lvert \bar{z} - \bar{x} \rvert^2},\bar{z})
        - (x_1+\sqrt{\varepsilon^2 - \lvert \bar{z} - \bar{x} \rvert^2})^{2\lambda_1}f(x_1+\sqrt{\varepsilon^2 - \lvert \bar{z} - \bar{x} \rvert^2},\bar{z}) \big \rvert \nonumber\\
        & \times \Big \lvert \frac{\partial}{\partial z_1} W_t^{\lambda_1}(x_1,x_1+\sqrt{\varepsilon^2 - \lvert \bar{z} - \bar{x} \rvert^2}) \Big \rvert \prod_{j=2}^n W_t^{\lambda_j}(x_j,z_j) \nonumber \\
        \leq &  C\frac{ \varepsilon}{t^{\lambda_1+3/2}}
         \prod_{j=2}^n \frac{1}{t^{\lambda_j+1/2}}  , \
         t \geq 1, \ \bar{z} \in K^{n-1} \text{ and } \lvert \bar{z} - \bar{x} \rvert < \varepsilon.
     \end{align}

     On the other hand, by \eqref{I3} we get
     \begin{align*}
       \frac{\partial^2}{\partial z_1^2} & W_t^{\lambda_1}(x_1,z_1) = \frac{e^{-(x_1^2+z_1^2)/4t}}{(2t)^{\lambda_1 + 1/2}}
       \left\{ \left( \frac{x_1}{2t} \right)^2 \left[ \left( \frac{x_1z_1}{2t}\right)^{-\lambda_1-1/2} I_{\lambda_1+1/2}\left( \frac{x_1z_1}{2t}\right) \right. \right. \\
        & + \left. \left( \frac{x_1z_1}{2t}\right)^2 \left( \frac{x_1z_1}{2t}\right)^{-\lambda_1-3/2} I_{\lambda_1+3/2}\left( \frac{x_1z_1}{2t}\right)
         - \frac{z_1^2}{2t} \left( \frac{x_1z_1}{2t}\right)^{-\lambda_1-1/2} I_{\lambda_1+1/2}\left( \frac{x_1z_1}{2t}\right) \right] \\
        &  -  \frac{1}{2t} \left( \frac{x_1z_1}{2t}\right)^{-\lambda_1+1/2} I_{\lambda_1-1/2}\left( \frac{x_1z_1}{2t}\right)\\
        &  - \left. \frac{z_1}{2t} \left[ \frac{x_1^2z_1}{4t^2} \left( \frac{x_1z_1}{2t}\right)^{-\lambda_1-1/2} I_{\lambda_1+1/2}\left( \frac{x_1z_1}{2t}\right)
         - \frac{z_1}{2t} \left( \frac{x_1z_1}{2t}\right)^{-\lambda_1+1/2} I_{\lambda_1-1/2}\left( \frac{x_1z_1}{2t}\right)\right]  \right\}, \ t,\,z_1 >0.
     \end{align*}

     Then, \eqref{I1} implies that
     \begin{align*}
       \Big\lvert \frac{\partial^2}{\partial z_1^2} & W_t^{\lambda_1}(x_1,z_1) \Big\rvert \leq C \frac{e^{-(x_1^2 + z_1^2)/4t}}{t^{\lambda_1 + 1/2}} \\
       & \times
         \Big(\dfrac{x_1^2}{t^2} \left[ 1 + \left( \dfrac{x_1z_1}{t}\right)^2 + \dfrac{z_1^2}{t}\right] +\frac{1}{t}+\frac{z_1}{t}\Big(\frac{x_1^2z_1}{t^2}+\frac{z_1}{t}\Big)\Big)\le \frac{C}{t^{\lambda_1+3/2}},\quad t \geq 1, \ z_1 \in K.
     \end{align*}
     Hence, it obtains
     \begin{align}\label{2.17}
       \big \lvert (x_1-&\sqrt{\varepsilon^2 - \lvert \bar{z} - \bar{x} \rvert^2})^{2\lambda_1} f(x_1-\sqrt{\varepsilon^2 - \lvert \bar{z} - \bar{x} \rvert^2},\bar{z}) \big \rvert \nonumber\\
        & \times \Big \lvert \frac{\partial}{\partial z_1} W_t^{\lambda_1}(x_1,x_1-\sqrt{\varepsilon^2 - \lvert \bar{z} - \bar{x} \rvert^2})
        - \frac{\partial}{\partial z_1} W_t^{\lambda_1}(x_1,x_1+\sqrt{\varepsilon^2 - \lvert \bar{z} - \bar{x} \rvert^2})\Big \rvert \prod_{j=2}^n W_t^{\lambda_j}(x_j,z_j) \nonumber \\
        \leq & C \frac{\varepsilon}{t^{\lambda_1 + 3/2}}
          \prod_{j=2}^n\frac{1}{t^{\lambda_j+1/2}}  , \
         t \geq 1, \ \bar{z} \in K^{n-1} \text{ and } \lvert \bar{z} - \bar{x} \rvert < \varepsilon.
     \end{align}

     From \eqref{2.16} and \eqref{2.17} we deduce that
     \begin{align}\label{2.18}
       \int_1^\infty & \lvert \phi(t) \rvert \int_{\bar{z}\in K^{n-1},\,\lvert \bar{z} - \bar{x} \rvert < \varepsilon }
        \big \lvert H_2(x,x_1-\sqrt{\varepsilon^2 - \lvert \bar{z} - \bar{x} \rvert^2},\bar{z},t)
        - H_2(x,x_1+\sqrt{\varepsilon^2 - \lvert \bar{z} - \bar{x} \rvert^2},\bar{z},t) \big \rvert \prod_{j=2}^n z_j^{2\lambda_j} d\bar{z} dt \nonumber\\
        \leq & C \varepsilon  \int_1^\infty \int_{\bar{z}\in K^{n-1},\,\lvert \bar{z} - \bar{x} \rvert < \varepsilon }
        \frac{1}{t^{\overset{n}{\underset{j=1}{\sum}}(\lambda_j + 1/2)+1}}
         \prod_{j=2}^n z_j^{2\lambda_j} d\bar{z} dt \longrightarrow 0, \,\,\hbox{as}\,\, \varepsilon \rightarrow 0^+.
     \end{align}

     By (\ref{I2}) and \eqref{2.11} it has,

     \begin{align*}
        \big \lvert (x_1-&\sqrt{\varepsilon^2 - \lvert \bar{z} - \bar{x} \rvert^2})^{2\lambda_1} f(x_1-\sqrt{\varepsilon^2 - \lvert \bar{z} - \bar{x} \rvert^2},\bar{z})
        - (x_1+\sqrt{\varepsilon^2 - \lvert \bar{z} - \bar{x} \rvert^2})^{2\lambda_1}f(x_1+\sqrt{\varepsilon^2 - \lvert \bar{z} - \bar{x} \rvert^2},\bar{z}) \big \rvert \\
        & \times \Big \lvert \frac{\partial}{\partial z_1} W_t^{\lambda_1}(x_1,x_1+\sqrt{\varepsilon^2 - \lvert \bar{z} - \bar{x} \rvert^2}) \Big \rvert \prod_{j=2}^n W_t^{\lambda_j}(x_j,z_j)  \\
        \leq & C \sqrt{\varepsilon^2 - \lvert \bar{z} - \bar{x} \rvert^2} \frac{1}{t^{(n+1)/2}} \prod_{j=2}^n (x_jz_j)^{-\lambda_j} e^{-\frac{|\bar{x}-\bar{z}|^2}{4t}-\frac{\varepsilon^2-|\bar{x}-\bar{z}|^2}{8t}} \\
        \leq & C \varepsilon \frac{1}{t^{(n+1)/2}}
        e^{-\frac{\varepsilon^2}{8t}},\,\,\,0<t<1,\,\,\,\bar{z} \in K^{n-1}\,\,\, and\,\,\,\lvert \bar{z} - \bar{x} \rvert <
        \varepsilon.
     \end{align*}

     Then \cite[Lemma 1.1]{ST} allows us to get

     \begin{align*}
     \int_0^1|\phi(t)| &\int_{|\bar{x}-\bar{z}|<\varepsilon}\big \lvert (x_1-\sqrt{\varepsilon^2 - \lvert \bar{z} - \bar{x} \rvert^2})^{2\lambda_1} f(x_1-\sqrt{\varepsilon^2 - \lvert \bar{z} - \bar{x} \rvert^2},\bar{z})
        \\ &- (x_1+\sqrt{\varepsilon^2 - \lvert \bar{z} - \bar{x} \rvert^2})^{2\lambda_1}f(x_1+\sqrt{\varepsilon^2 - \lvert \bar{z} - \bar{x} \rvert^2},\bar{z}) \big \rvert \Big \lvert \frac{\partial}{\partial z_1} W_t^{\lambda_1}(x_1,x_1+\sqrt{\varepsilon^2 - \lvert \bar{z} - \bar{x} \rvert^2}) \Big \rvert \\
        & \times \prod_{j=2}^n W_t^{\lambda_j}(x_j,z_j) \prod_{j=2}^n z_j^{2\lambda_j}d\bar{z} dt \\
        &\leq C\varepsilon \int_{|\bar{x}-\bar{z}|<\varepsilon}\int_0^1\frac{e^{-\varepsilon^2/8t}}{t^{(n+1)/2}}dtd\bar{z}\\
        &\le C\varepsilon \longrightarrow 0, \quad \text{ as } \varepsilon \rightarrow 0^+.\\
     \end{align*}

     Also, we write, for each $0<t<1$, $\bar{z} \in K^{n-1}$ and $\lvert \bar{z} - \bar{x} \rvert < \varepsilon,$
     \begin{align*}
      (x_1-&\sqrt{\varepsilon^2 - \lvert \bar{z} - \bar{x} \rvert^2})^{2\lambda_1} f(x_1-\sqrt{\varepsilon^2 - \lvert \bar{z} - \bar{x} \rvert^2},\bar{z})
      \left(\frac{\partial}{\partial z_1} W_t^{\lambda_1}(x_1,x_1-\sqrt{\varepsilon^2 - \lvert \bar{z} - \bar{x} \rvert^2}) \right. \\
      &-\left.\frac{\partial}{\partial z_1} W_t^{\lambda_1}(x_1,x_1+\sqrt{\varepsilon^2 - \lvert \bar{z} - \bar{x} \rvert^2})\right)\prod_{j=2}^n W_t^{\lambda_j}(x_j,z_j)\\
      =& \left((x_1-\sqrt{\varepsilon^2 - \lvert \bar{z} - \bar{x} \rvert^2})^{2\lambda_1}
      f(x_1-\sqrt{\varepsilon^2 - \lvert \bar{z} - \bar{x} \rvert^2},\bar{z})-x_1^{2\lambda_1}f(x)\right)
      \left(\frac{\partial}{\partial z_1} W_t^{\lambda_1}(x_1,x_1-\sqrt{\varepsilon^2 - \lvert \bar{z} - \bar{x} \rvert^2}) \right.
      \\&-\left.\frac{\partial}{\partial z_1} W_t^{\lambda_1}(x_1,x_1+\sqrt{\varepsilon^2 - \lvert \bar{z} - \bar{x} \rvert^2})\right)
      \prod_{j=2}^n W_t^{\lambda_j}(x_j,z_j)+ x_1^{2\lambda_1}f(x)\left(\frac{\partial}{\partial z_1}
      W_t^{\lambda_1}(x_1,x_1-\sqrt{\varepsilon^2 - \lvert \bar{z} - \bar{x} \rvert^2}) \right. \\
      &-\left.\frac{\partial}{\partial z_1}  W_t^{\lambda_1}(x_1,x_1+\sqrt{\varepsilon^2 - \lvert \bar{z} - \bar{x} \rvert^2})\right)\prod_{j=2}^n W_t^{\lambda_j}(x_j,z_j).
     \end{align*}

     By \eqref{2.11} we have
     $$
     \Big|\frac{\partial}{\partial z_1} W_t^{\lambda_1}(x_1,z_1)\Big|
     \le C\frac{e^{-(x_1-z_1)^2/8t}}{t}, \,\,\,0<t<1,\,\,\,z_1\in K.
     $$

     By proceeding as above we obtain,
      \begin{align*}
     \int_0^1&|\phi(t)| \int_{\bar{z}\in K^{n-1},\,\,|\bar{x}-\bar{z}|<\varepsilon}\big \lvert (x_1-\sqrt{\varepsilon^2 - \lvert \bar{z} - \bar{x} \rvert^2})^{2\lambda_1} f(x_1-\sqrt{\varepsilon^2 - \lvert \bar{z} - \bar{x} \rvert^2},\bar{z})-x_1^{2\lambda_1}f(x)\big \rvert  \\
     &\times \left|\frac{\partial}{\partial z_1} W_t^{\lambda_1}(x_1,x_1-\sqrt{\varepsilon^2 - \lvert \bar{z} - \bar{x} \rvert^2}) -\frac{\partial}{\partial z_1} W_t^{\lambda_1}(x_1,x_1+\sqrt{\varepsilon^2 - \lvert \bar{z} - \bar{x} \rvert^2})\right|\prod_{j=2}^n W_t^{\lambda_j}(x_j,z_j) \prod_{j=2}^n z_j^{2\lambda_j}d\bar{z}\\
         \leq & C\varepsilon \int_{\bar{z}\in K^{n-1},\,\,|\bar{x}-\bar{z}|<\varepsilon}\int_0^1 \frac{e^{-\varepsilon^2/8t}}{t^{(n+1)/2}}dtd\bar{z}\\
          \leq&  C\varepsilon \longrightarrow 0, \quad \text{ as } \varepsilon \rightarrow 0^+.
     \end{align*}

     From the above estimates we conclude that
     \begin{align*}
     \int_0^\infty & \phi(t)\int_{(0,\infty)^n}\Delta_{\lambda_1,z_1}f(z)\Big(\prod_{j=1}^nW_t^{\lambda_j}(x_j,z_j)-\frac{\chi_{(1,\infty)}(t)t^{-\overset{n}{\underset{j=1}{\sum}}(\lambda_j+1/2)}}{\overset{n}{\underset{j=1}{\prod}}2^{2\lambda_j}\Gamma(\lambda_j+1/2)}\Big)\prod_{j=1}^nz_j^{2\lambda_j}dzdt\\
        &=-\lim_{\varepsilon \rightarrow 0^+} \left[ \int_0^\infty \phi(t) \int_{|x-z|>\varepsilon} f(z) \frac{\partial}{\partial t} \Big(W_t^{\lambda_1}(x_1,z_1)\Big) \prod_{j=2}^n W_t^{\lambda_j}(x_j,z_j) \prod_{j=1}^n z_j^{2\lambda_j}dzdt \right. \\
        &+ \left.f(x)x_1^{2\lambda_1} \int_0^1 \phi(t) \int_{|\bar{x}-\bar{z}|<\varepsilon} \int_{x_1-\sqrt{\varepsilon^2 -|\bar{x}-\bar{z}|^2}}^{x_1+\sqrt{\varepsilon^2 -|\bar{x}-\bar{z}|^2}} \frac{\partial^2}{\partial z_1^2}W_t^{\lambda_1}(x_1,z_1)\prod_{j=2}^n W_t^{\lambda_j}(x_j,z_j) dz_1 \prod_{j=2}^n z_j^{2\lambda_j}d\bar{z} dt \right]\\
        =&- \lim_{\varepsilon \rightarrow 0^+} \left[ \int_0^\infty \phi(t) \int_{|x-z|>\varepsilon} f(z) \frac{\partial}{\partial t}\Big(W_t^{\lambda_1}(x_1,z_1)\Big) \prod_{j=2}^n W_t^{\lambda_j}(x_j,z_j) \prod_{j=1}^n z_j^{2\lambda_j}dzdt \right. \\
        &+ \left.f(x)x_1^{2\lambda_1} \int_0^1 \phi(t) \int_{|x-z|<\varepsilon}  \frac{\partial^2}{\partial z_1^2}W_t^{\lambda_1}(x_1,z_1)\prod_{j=2}^n W_t^{\lambda_j}(x_j,z_j) \prod_{j=2}^n z_j^{2\lambda_j}dz dt \right].
     \end{align*}

     We now denote by $\mathbb{W}_t(u,v)$ the classical heat kernel, that is,
     $$
     \mathbb{W}_t(u,v)=\frac{1}{2\sqrt{\pi t}}e^{-|u-v|^2/4t},\,\,\,u,v\in \mathbb{R},\,\,t>0.
     $$
     According to \eqref{I3} we get
     \begin{align*}
        \frac{\partial^2}{\partial z_1^2}&W_t^{\lambda_1}(x_1,z_1) = \frac{\partial^2}{\partial z_1^2} \left[ \frac{\sqrt{2\pi}}{(2t)^{\lambda_1}} \left( \frac{x_1z_1}{2t}\right)^{-\lambda_1+1/2}I_{\lambda_1-1/2}\left( \frac{x_1z_1}{2t}\right)e^{-x_1z_1/2t}\mathbb{W}_t(x_1,z_1)\right]\\
        =& \frac{\sqrt{2\pi}}{(2t)^{\lambda_1}} \left\{\frac{\partial^2}{\partial z_1^2} \left[\left( \frac{x_1z_1}{2t}\right)^{-\lambda_1+1/2}I_{\lambda_1-1/2}\left( \frac{x_1z_1}{2t}\right)e^{-x_1z_1/2t}\right]\mathbb{W}_t(x_1,z_1)\right.\\
        & +2 \frac{\partial}{\partial z_1} \left[\left( \frac{x_1z_1}{2t}\right)^{-\lambda_1+1/2}I_{\lambda_1-1/2}\left( \frac{x_1z_1}{2t}\right)e^{-x_1z_1/2t}\right]\frac{\partial}{\partial z_1}\mathbb{W}_t(x_1,z_1)\\
        & + \left. \left( \frac{x_1z_1}{2t}\right)^{-\lambda_1+1/2}I_{\lambda_1-1/2}\left( \frac{x_1z_1}{2t}\right)e^{-x_1z_1/2t}\frac{\partial^2}{\partial z_1^2}\mathbb{W}_t(x_1,z_1) \right\}\\
        =& \frac{\sqrt{2\pi}}{(2t)^{\lambda_1}}e^{-{x_1z_1/2t}}\left(\frac{x_1z_1}{2t}\right)^{-\lambda_1}\left\{ \left[ \left( \frac{x_1z_1}{2t}\right)^{1/2} I_{\lambda_1+3/2}\left( \frac{x_1z_1}{2t}\right) \right. \right.\\
        &+ \left( \frac{x_1z_1}{2t}\right)^{1/2} I_{\lambda_1+1/2}\left( \frac{x_1z_1}{2t}\right) \frac{2t}{x_1z_1} - 2\left( \frac{x_1z_1}{2t}\right)^{1/2} I_{\lambda_1+1/2}\left( \frac{x_1z_1}{2t}\right)\\
        &+ \left.\left( \frac{x_1z_1}{2t}\right)^{1/2} I_{\lambda_1-1/2}\left( \frac{x_1z_1}{2t}\right) \right]\mathbb{W}_t(x_1,z_1) \left( \frac{x_1}{2t}\right)^2\\
        &+\frac{x_1}{t}\left[ \left(\frac{x_1z_1}{2t}\right)^{1/2} I_{\lambda_1+1/2}\left( \frac{x_1z_1}{2t}\right) - \left(\frac{x_1z_1}{2t}\right)^{1/2} I_{\lambda_1-1/2}\left( \frac{x_1z_1}{2t}\right)\right] \frac{\partial}{\partial z_1}\mathbb{W}_t(x_1,z_1) \\
        &+ \left. \left[\left(\frac{x_1z_1}{2t}\right)^{1/2} I_{\lambda_1-1/2}\left( \frac{x_1z_1}{2t}\right)\right]\frac{\partial^2}{\partial z_1^2} \mathbb{W}_t(x_1,z_1) \right\}, \quad t,z_1\in (0,\infty).
     \end{align*}

     By using \eqref{I2} we obtain for every $0<t<1$, and $z_1\in K$,
     \begin{align}\label{F1}
       \bullet\,\,\, \left| \frac{\sqrt{2\pi}}{(2t)^{\lambda_1}}\left( \frac{x_1z_1}{2t}\right)^{-\lambda_1+1/2}\right. & \left.I_{\lambda_1-1/2}\left( \frac{x_1z_1}{2t}\right)e^{-x_1z_1/2t}\frac{\partial^2}{\partial z_1^2}\mathbb{W}_t(x_1,z_1)- (x_1z_1)^{-\lambda_1}\frac{\partial^2}{\partial z_1^2} \mathbb{W}_t(x_1,z_1) \right| \nonumber\\
        \leq & Ct\left| \frac{\partial^2}{\partial z_1^2} \mathbb{W}_t(x_1,z_1)\right|,
     \end{align}

     \begin{align}\label{F2}
        \bullet \,\,\,\left| \frac{\sqrt{2\pi}}{(2t)^{\lambda_1}}\right. & \left( \frac{x_1z_1}{2t}\right)^{-\lambda_1} \frac{x_1}{2t}
        \frac{\partial}{\partial z_1}\mathbb{W}_t(x_1,z_1)e^{-x_1z_1/2t}\left[\left( \frac{x_1z_1}{2t}\right)^{1/2}I_{\lambda_1+1/2}
        \left( \frac{x_1z_1}{2t}\right)\right.\nonumber\\
        &- \left.\left.\left( \frac{x_1z_1}{2t}\right)^{1/2}I_{\lambda_1-1/2}
        \left( \frac{x_1z_1}{2t}\right) \right]\right| \leq  C\left| \frac{\partial}{\partial z_1} \mathbb{W}_t(x_1,z_1)\right|;
     \end{align}

     and

     \begin{align}\label{F3}
     \bullet\,\,\,&\left| \frac{\sqrt{2\pi}}{(2t)^{\lambda_1}} e^{-x_1z_1/2t} \left( \frac{x_1z_1}{2t}\right)^{-\lambda_1} \mathbb{W}_t(x_1,z_1) \left( \frac{x_1}{2t}\right)^2 \left(
     \left( \frac{x_1z_1}{2t}\right)^{1/2} I_{\lambda_1+3/2}\left( \frac{x_1z_1}{2t}\right) \right.\right.\nonumber \\
        &+ \left( \frac{x_1z_1}{2t}\right)^{1/2} I_{\lambda_1+1/2}\left( \frac{x_1z_1}{2t}\right) \frac{2t}{x_1z_1} - 2\left( \frac{x_1z_1}{2t}\right)^{1/2} I_{\lambda_1+1/2}\left( \frac{x_1z_1}{2t}\right)\nonumber\\
        &+ \left.\left.\left( \frac{x_1z_1}{2t}\right)^{1/2} I_{\lambda_1-1/2}\left( \frac{x_1z_1}{2t}\right)\right)\right|\nonumber\\
        &\leq C  \mathbb{W}_t(x_1,z_1).
     \end{align}

     Hence, we get for every $0<t<1$, and $z_1\in K$,
     $$ \left| \frac{\partial^2}{\partial z_1^2}W_t^{\lambda_1 }(x_1,z_1) - (x_1z_1)^{-\lambda_1}\frac{\partial^2}{\partial z_1^2} \mathbb{W}_t(x_1,z_1)\right|
     \leq C\frac{e^{-(x_1-z_1)^2/8t}}{t}.$$

     Then, by involving again \eqref{I2} and \cite[Lemma 1.1]{ST} we have that
     \begin{align*}
     \int_0^1& |\phi(t)| \int_{|x-z|<\varepsilon} \left| \frac{\partial^2}{\partial z_1^2}W_t^{\lambda_1 }(x_1,z_1) - (x_1z_1)^{-\lambda_1}\frac{\partial^2}{\partial z_1^2} \mathbb{W}_t(x_1,z_1)\right|
      \prod_{j=2}^n W_t^{\lambda_j}(x_j,z_j) \prod_{j=2}^n z_j^{2\lambda_j}dzdt \\
     \leq & C \int_{|x-z|<\varepsilon} \int_0^1  \frac{e^{-\lvert x-z \rvert ^2/8t}}{t^{(n+1)/2}} \prod_{j=2}^n z_j^{2\lambda_j} dt dz\\
     \leq & C\int_{|x-z|<\varepsilon} \frac{dz}{|x-z|^{n-1}}\le C\varepsilon
      \longrightarrow 0, \text{ as } \varepsilon \rightarrow 0^+.
     \end{align*}

     Also, \eqref{I2} leads to, for every $0<t<1$, and $z\in K^n$,
    \begin{align*}
       \Big\lvert \prod_{j=1}^{i-1} & (x_jz_j)^{-\lambda_j} \frac{\partial^2}{\partial z_1^2} \mathbb{W}_t(x_1,z_1) \prod_{j=2}^{i-1} \mathbb{W}_t(x_j,z_j)
       \left( W_t^{\lambda_i}(x_i,z_i) - (x_iz_i)^{-\lambda_i}\mathbb{W}_t(x_i,z_i) \right) \\
       & \times  \prod_{j=i+1}^{n} W_t^{\lambda_j}(x_j,z_j) \Big \rvert \leq C \frac{e^{- \lvert x-z \rvert^2/8t}}{t^{n/2}}, \quad i=2, \dots, n.
     \end{align*}
     Then, we have that for every $i=2, \dots, n,$
     \begin{align*}
       \int_0^1 \lvert \phi(t)& \rvert \int_{\lvert x -z \rvert < \varepsilon}
       \Big\lvert \prod_{j=1}^{i-1}  (x_jz_j)^{-\lambda_j} \frac{\partial^2}{\partial z_1^2} \mathbb{W}_t(x_1,z_1) \prod_{j=2}^{i-1} \mathbb{W}_t(x_j,z_j)
       \left( W_t^{\lambda_i}(x_i,z_i) - (x_iz_i)^{-\lambda_i}W_t(x_i,z_i) \right) \\
       & \times  \prod_{j=i+1}^{n} W_t^{\lambda_j}(x_j,z_j) \Big \rvert \prod_{j=2}^n z_j^{2\lambda_j} dz dt \\
        \leq & C  \int_{\lvert x -z \rvert < \varepsilon} \int_0^1  \frac{e^{- \lvert x-z \rvert^2/8t}}{t^{n/2}} dt dz \\
        \leq & C  \int_{\lvert x -z \rvert < \varepsilon} \frac{dz}{\lvert x - z \rvert^{n-2}} \longrightarrow 0, \quad \text{as } \varepsilon \rightarrow 0^+,
     \end{align*}
     provided that $n>2$. When $n=2$ we proceed in a similar way.

     Hence, we conclude that
     \begin{align*}
     \int_0^\infty& \phi(t)\int_{(0,\infty)^n}\Delta_{\lambda_1,z_1}f(z)\Big(\prod_{j=1}^nW_t^{\lambda_j}(x_j,z_j)-\frac{\chi_{(1,\infty)}(t)t^{-\overset{n}{\underset{j=1}{\sum}}(\lambda_j+1/2)}}{\overset{n}{\underset{j=1}{\prod}}2^{2\lambda_j}\Gamma(\lambda_j+1/2)}\Big)\prod_{j=1}^nz_j^{2\lambda_j}dz_jdt\\
          &= -\lim_{\varepsilon \rightarrow 0^+} \left[ \int_0^\infty \phi(t) \int_{|x-z|>\varepsilon} f(z) \frac{\partial}{\partial t} \Big(W_t^{\lambda_1}(x_1,z_1)\Big) \prod_{j=2}^n W_t^{\lambda_j}(x_j,z_j)\prod_{j=1}^n z_j^{2\lambda_j}dzdt \right. \\
         &+ \left.f(x)x_1^{2\lambda_1} \int_0^1 \phi(t) \int_{|x-z|<\varepsilon}  \prod_{j=1}^n (x_jz_j)^{-\lambda_j}
            \frac{\partial^2}{\partial z_1^2}\mathbb{W}_t(x_1,z_1)\prod_{j=2}^n \mathbb{W}_t(x_j,z_j) \prod_{j=2}^n z_j^{2\lambda_j}dz dt \right].
     \end{align*}

     On the other hand, the mean value theorem allows us to write, for every $\alpha \in \mathbb{R}$,
     and $j=1, \dots, n$,
     \begin{align*}
       \lvert z_j^\alpha - x_j^\alpha \rvert & \ \left \lvert \frac{\partial^2}{\partial z_1^2}\mathbb{W}_t(x_1,z_1)\prod_{j=2}^n \mathbb{W}_t(x_j,z_j) \right \rvert
           \leq C \lvert z_j - x_j \rvert \frac{e^{-\lvert x- z \rvert^2/8t}}{t^{n/2+1}} \\
         & \leq C \frac{e^{-\lvert x- z \rvert^2/16t}}{t^{(n+1)/2}},\,\,\,t>0\,\,and\,\,z\in K^n.
     \end{align*}

     Then, by proceeding as above we obtain
     \begin{align} \label{2.21}
        \int_0^\infty& \phi(t)\int_{(0,\infty)^n}\Delta_{\lambda_1,z_1}f(z)\Big(\prod_{j=1}^nW_t^{\lambda_j}(x_j,z_j)-\frac{\chi_{(1,\infty)}(t)t^{-\overset{n}{\underset{j=1}{\sum}}(\lambda_j+1/2)}}{\overset{n}{\underset{j=1}{\prod}}2^{2\lambda_j}\Gamma(\lambda_j+1/2)}\Big)\prod_{j=1}^nz_j^{2\lambda_j}dz_jdt\nonumber\\
          &= -\lim_{\varepsilon \rightarrow 0^+} \left[ \int_0^\infty \phi(t) \int_{|x-z|>\varepsilon} f(z) \frac{\partial}{\partial t}\Big(W_t^{\lambda_1}(x_1,z_1)\Big)\prod_{j=2}^n W_t^{\lambda_j}(x_j,z_j) \prod_{j=1}^n z_j^{2\lambda_j}dzdt \right. \nonumber\\
         &+ \left.f(x) \int_0^1 \phi(t) \int_{|x-z|<\varepsilon}
            \frac{\partial^2}{\partial z_1^2} \frac{e^{- \lvert x-z \rvert^2/4t}}{(2 \sqrt{\pi t})^n} dz dt \right]\nonumber\\
         = & -\lim_{\varepsilon \rightarrow 0^+} \left[ \int_0^\infty \phi(t) \int_{|x-z|>\varepsilon} f(z) \frac{\partial}{\partial t}\Big(W_t^{\lambda_1}(x_1,z_1)\Big) \prod_{j=2}^n W_t^{\lambda_j}(x_j,z_j)\prod_{j=1}^n z_j^{2\lambda_j}dzdt \right. \nonumber\\
         &+ \left.f(x) \int_0^1 \phi(t) \int_{|y|<\varepsilon}
            \frac{\partial^2}{\partial y_1^2} \frac{e^{- \lvert y \rvert^2/4t}}{(2 \sqrt{\pi t})^n} dy dt \right].
     \end{align}

     Also, we have that, for every $i=2,\ldots,n$,
     \begin{align} \label{YZ}
     \int_0^\infty& \phi(t)\int_{(0,\infty)^n}\Delta_{\lambda_i,z_i}f(z)\Big(\prod_{j=1}^nW_t^{\lambda_j}(x_j,z_j)-\frac{\chi_{(1,\infty)}(t)t^{-\overset{n}{\underset{j=1}{\sum}}(\lambda_j+1/2)}}{\overset{n}{\underset{j=1}{\prod}}2^{2\lambda_j}\Gamma(\lambda_j+1/2)}\Big)\prod_{j=1}^nz_j^{2\lambda_j}dz_jdt\nonumber\\
          = & -\lim_{\varepsilon \rightarrow 0^+} \left[ \int_0^\infty \phi(t) \int_{|x-z|>\varepsilon} f(z) \frac{\partial}{\partial t}\Big(W_t^{\lambda_i}(x_i,z_i)\Big) \prod_{j=1,\,j\neq i}^n W_t^{\lambda_j}(x_j,z_j)\prod_{j=1}^n z_j^{2\lambda_j}dzdt \right. \\
         &+ \left.f(x) \int_0^1 \phi(t) \int_{|y|<\varepsilon}
            \frac{\partial^2}{\partial y_i^2} \frac{e^{- \lvert y \rvert^2/4t}}{(2 \sqrt{\pi t})^n} dy dt \right]. \nonumber
     \end{align}

     Then, by \eqref{2.21} and \eqref{YZ}, it follows that
     \begin{align}\label{XYZ}
     T&_{\lambda_1,\ldots,\lambda_n}^mf(x)=-\lim_{\varepsilon \rightarrow 0^+} \left[ \int_0^\infty \phi(t) \int_{|x-z|>\varepsilon} f(z) \frac{\partial}{\partial t}\Big(\prod_{j=1}^n W_t^{\lambda_j}(x_j,z_j)\Big)\prod_{j=1}^n z_j^{2\lambda_j}dzdt \right. \\
         &+ \left.nf(x) \int_0^1 \phi(t) \int_{|y|<\varepsilon}
            \frac{\partial^2}{\partial y_1^2} \frac{e^{- \lvert y \rvert^2/4t}}{(2 \sqrt{\pi t})^n} dy dt \right]. \nonumber
     \end{align}

     We define
     $$\alpha(\varepsilon) = \int_0^1 \phi(t) \int_{|y|<\varepsilon} \frac{\partial^2}{\partial y_1^2} \frac{e^{- \lvert y \rvert^2/4t}}{(2 \sqrt{\pi t})^n} dy dt, \quad \varepsilon>0.$$

     We can write
     \begin{align*}
       \alpha(\varepsilon) & = \int_0^1 \phi(t) \int_{|\bar{y}|<\varepsilon} \int_{-\sqrt{\varepsilon^2-\lvert \bar{y} \rvert^2}}^{\sqrt{\varepsilon^2-\lvert \bar{y} \rvert^2}}
                               \frac{\partial^2}{\partial y_1^2} \frac{e^{- \lvert y \rvert^2/4t}}{(2 \sqrt{\pi t})^n} dy_1 d\bar{y} dt \\
                           & = \int_0^1 \phi(t) \int_{|\bar{y}|<\varepsilon}
                               \left[ \frac{\partial}{\partial y_1} \left( \frac{e^{- (y_1^2 + \lvert \bar{y} \rvert^2)/4t}}{(2 \sqrt{\pi t})^n} \right)_{\left\lvert y_1=\sqrt{\varepsilon^2-\lvert \bar{y} \rvert^2} \right. }
                                    - \frac{\partial}{\partial y_1} \left( \frac{e^{- (y_1^2 + \lvert \bar{y} \rvert^2)/4t}}{(2 \sqrt{\pi t})^n} \right)_{\left\lvert y_1=-\sqrt{\varepsilon^2-\lvert \bar{y} \rvert^2} \right. } \right] d\bar{y} dt, \ \varepsilon>0.
     \end{align*}
     Hence, \cite[Lemma 1.1]{ST} leads to
     \begin{align*}
       \lvert \alpha (\varepsilon) \rvert & \leq C \int_0^1 \int_{\lvert \bar{y} \rvert < \varepsilon} \frac{e^{-\varepsilon^2/8t}}{t^{(n+1)/2}} d\bar{y} dt
                                           \leq C  \int_{\lvert \bar{y} \rvert < \varepsilon} \int_0^1 \frac{e^{-\varepsilon^2/8t}}{t^{(n+1)/2}} dt d\bar{y}
                                           \leq \frac{1}{\varepsilon^{n-1}} \int_{\lvert \bar{y} \rvert < \varepsilon} d\bar{y} \leq C, \quad \varepsilon>0.
     \end{align*}

     Suppose now that there exists  $\displaystyle \phi(0^+)=\lim_{t \rightarrow 0^+} \phi(t)$. Then, we have that
     \begin{equation} \label{M1}
       \lim_{\varepsilon \rightarrow 0^+} \int_0^1 \phi(t) \int_{\lvert y \rvert < \varepsilon}
       \frac{\partial^2}{\partial y_1^2} \frac{e^{-\lvert y \rvert^2/4t}}{(2\sqrt{\pi t})^n} dy dt = - M \phi(0^+),
     \end{equation}
     for a certain $M>0$. Indeed, by making changes of variables we obtain
     \begin{align*}
       \int_0^1 \phi(t) \int_{\lvert y \rvert < \varepsilon} \frac{\partial^2}{\partial y_1^2} \frac{e^{-\lvert y \rvert^2/4t}}{(2\sqrt{\pi t})^n} dy dt
        & = \int_0^{1/\varepsilon^2} \phi(s\varepsilon^2)  \int_{\lvert z \rvert < 1} \frac{\partial^2}{\partial z_1^2} \frac{e^{-\lvert z \rvert^2/4s}}{(2\sqrt{\pi s})^n} dz ds\\
        & = \int_0^{1/\varepsilon^2} \phi(s\varepsilon^2)  \int_{\lvert \bar{z} \rvert < 1}
            \left. \frac{\partial}{\partial z_1} \frac{e^{-\lvert z \rvert^2/4s}}{(2\sqrt{\pi s})^n} \right]_{z_1=-\sqrt{1-\lvert \bar{z} \rvert^2}}^{z_1=\sqrt{1-\lvert \bar{z} \rvert^2}} d\bar{z} ds\\
        & = -\int_0^{1/\varepsilon^2} \phi(s\varepsilon^2)  \int_{\lvert \bar{z} \rvert < 1}
            \left. \frac{z_1}{2s} \frac{e^{-\lvert z \rvert^2/4s}}{(2 \sqrt{\pi s})^n} \right]_{z_1=-\sqrt{1-\lvert \bar{z} \rvert^2}}^{z_1=\sqrt{1-\lvert \bar{z} \rvert^2}} d\bar{z} ds\\
        & = - \int_0^{1/\varepsilon^2} \phi(s\varepsilon^2)  \int_{\lvert \bar{z} \rvert < 1}
            \frac{\sqrt{1-\lvert \bar{z} \rvert^2}}{s} \frac{e^{-1/4s}}{(2\sqrt{\pi s})^n} d\bar{z} ds\\
        & = - M \int_0^{1/\varepsilon^2} \phi(s\varepsilon^2) \frac{e^{-1/4s}}{s^{n/2+1}} ds, \quad \varepsilon>0,
     \end{align*}
     where
     $$M=\frac{1}{(2\sqrt{\pi})^n} \int_{\lvert \bar{z} \rvert < 1} \sqrt{1- \lvert \bar{z} \rvert^2} d\bar{z}.$$
     Moreover, denoting by $\chi_{[0,1]}$ the characteristic function of $[0,1]$, we have that
     $$\int_0^{1/\varepsilon^2} \phi(s\varepsilon^2) \frac{e^{-1/4s}}{s^{n/2+1}} ds
     = \int_0^\infty \phi(s \varepsilon^2) \chi_{[0,1]}(s\varepsilon^2)\frac{e^{-1/4s}}{s^{n/2+1}} ds, \quad \varepsilon>0.$$
     Then, by using the dominated convergence theorem, it follows that
     $$\lim_{\varepsilon\to 0^+}\int_0^{1/\varepsilon^2} \phi(s\varepsilon^2) \frac{e^{-1/4s}}{s^{n/2+1}} ds = \phi(0^+) \int_0^\infty  \frac{e^{-1/4s}}{s^{n/2+1}} ds.$$
     Hence, there exists $M>0$ for which \eqref{M1} holds.\\

     Let $\alpha>-1/2$. According to \eqref{I3} it follows that
     \begin{align}\label{Z0}
     \frac{\partial}{\partial t}&W_t^\alpha(u,v)=\frac{e^{-(u^2+v^2)/4t}}{2^{\alpha+1/2}}\Big\{-\frac{\alpha+1/2}{t^{\alpha+3/2}}\Big(\frac{uv}{2t}\Big)^{-\alpha+1/2}I_{\alpha-1/2}\Big(\frac{uv}{2t}\Big)\nonumber\\
     &-\frac{uv}{2t^{\alpha+5/2}}\Big(\frac{uv}{2t}\Big)^{-\alpha+1/2}I_{\alpha+1/2}\Big(\frac{uv}{2t}\Big)+\frac{u^2+v^2}{4t^{\alpha+5/2}}\Big(\frac{uv}{2t}\Big)^{-\alpha+1/2}I_{\alpha-1/2}\Big(\frac{uv}{2t}\Big)\Big\},\,\,\,u,v,t\in (0,\infty).
     \end{align}
     Then, by \eqref{I1} we obtains
     \begin{equation}\label{Z1}
     \Big|\frac{\partial}{\partial t}W_t^\alpha(u,v)\Big|\le C\frac{1}{t^{\alpha+3/2}},\,\,\,\frac{uv}{t}\le 1,\,\,\,u,v,t\in (0,\infty),
     \end{equation}
     and, by \eqref{I2} we have
     \begin{align}\label{ABC}
     \Big|\frac{\partial}{\partial t}&W_t^\alpha(u,v)\Big|\le C\frac{e^{-(u-v)^2/4t}}{t^{3/2}}(uv)^{-\alpha}\Big(\frac{|u-v|^2}{t}+1+\frac{u^2+v^2}{uv}\Big)\nonumber\\
     &\le C\frac{e^{-(u-v)^2/4t}}{t^{3/2}}(uv)^{-\alpha}\Big(1+\frac{u^2+v^2}{uv}\Big),\,\,\,\frac{uv}{t}\ge 1,\,\,\,u,v,t\in (0,\infty).
     \end{align}

          For every $t>0$ and $z \in (0,\infty)^n$,
     \begin{align*}
       \frac{\partial}{\partial t}  \prod_{j=1}^n &W_t^{\lambda_j}(x_j,z_j)
         = \sum_{i=1}^n  \prod_{\tiny{\begin{array}{c} j=1 \\ j \neq i \end{array}}}^n W_t^{\lambda_j}(x_j,z_j)\frac{\partial}{\partial t} W_t^{\lambda_i}(x_i,z_i)\\
         = & \sum_{i=1}^n  \prod_{\tiny{\begin{array}{c} j=1 \\ j \neq i \end{array}}}^n W_t^{\lambda_j}(x_j,z_j)\frac{e^{-(x_i^2+z_i^2)/4t}}{2^{\lambda_i+1/2}}
         \left\{ -\frac{\lambda_i + 1/2}{t^{\lambda_i+3/2}} \left( \frac{x_iz_i}{2t}\right)^{-\lambda_i+1/2}I_{\lambda_i-1/2}\left( \frac{x_iz_i}{2t}\right) \right. \\
         & - \left. \frac{x_i^2z_i^2}{4t^{\lambda_i+7/2}} \left( \frac{x_iz_i}{2t}\right)^{-\lambda_i-1/2}I_{\lambda_i+1/2}\left( \frac{x_iz_i}{2t}\right)
          +  \frac{x_i^2+z_i^2}{4t^{\lambda_i+5/2}} \left( \frac{x_iz_i}{2t}\right)^{-\lambda_i+1/2}I_{\lambda_i-1/2}\left( \frac{x_iz_i}{2t}\right) \right\}.
     \end{align*}

     Hence, by using \eqref{I1}, \eqref{I2}, \eqref{Z1}, and \ref{ABC}, we get, for every $z\in K^n$,
     \begin{align*}
     \Big|\frac{\partial}{\partial t}&\prod_{j=1}^n W_t^{\lambda_j}(x_j,z_j)\Big|\le \sum_{i=1}^n\Big|\frac{\partial}{\partial t}W_t^{\lambda_i}(x_i,z_i)\Big|\prod_{j=1,\,j\neq i}^n W_t^{\lambda_j}(x_j,z_j)\\
     &\le C\left\{\begin{array}{lcl}
           t^{-\overset{n}{\underset{j=1}{\sum}}(\lambda_j+1/2)+1} &,& t>1\\
         &&\\
          e^{-|x-z|^2/8t}t^{-{(n+2)/2}}   &,& 0<t\le 1.
       \end{array}  \right. \\
       \end{align*}

     Hence, since $f \in C_c^\infty((0,\infty)^n)$, it follows that, for every $\varepsilon >0$
    $$
    \int_0^\infty |\phi(t)|\int_{(0,\infty)^n,\,|x-z|>\varepsilon}|f(z)|\Big|\frac{\partial}{\partial t}\Big(\prod_{j=1}^n W_t^{\lambda_j}(x_j,z_j)\Big|\prod_{j=1}^nz_j^{2\lambda_j}dzdt<\infty.
    $$

The we can interchange the order of integration on the integrals in (\ref{XYZ}) and by using (\ref{M1}) we conclude that
     $$T_{\lambda_1, \dots, \lambda_n}^m f(x)  = -
         \lim_{\varepsilon \rightarrow 0^+}\int_{\tiny{\begin{array}{l}(0,\infty)^n \\ \lvert z-x \rvert > \varepsilon \end{array}}} f(z) \int_0^\infty \phi(t)
          \frac{\partial}{\partial t} \prod_{j=1}^n W_t^{\lambda_j}(x_j,z_j)dt  \prod_{j=1}^n z_j^{2\lambda_j} dz  +C \phi(0^+)f(x),$$
          for a certain $C>0$.
    \end{proof}

   We now prove the $L^p$-boundedness properties of the maximal operator associated with the principal value
   integral that appears in Proposition~\ref{integral representation}.

    \begin{Prop}\label{maximal}
     Let $\lambda_j >-1/2$, $j=1, \dots, n$. The maximal operator $T_{\lambda_1, \dots, \lambda_n}^{m,*}$ defined by
     $$T_{\lambda_1, \dots, \lambda_n}^{m,*}(f)(x)=\sup_{\varepsilon>0}
       \left \lvert \int_{(0,\infty)^n, \ \lvert x-y \rvert>\varepsilon}
       f(y) K_{\lambda_1, \dots, \lambda_n}^\phi (x,y) \prod_{j=1}^n y_j^{2\lambda_j} dy \right \rvert, \ x \in (0,\infty)^n,$$
     where
     $$K_{\lambda_1, \dots, \lambda_n}^\phi (x,y)= \int_0^\infty \phi(t) \frac{\partial}{\partial t} \prod_{j=1}^n W_t^{\lambda_j}(x_j,y_j) dt, \quad x,y \in (0,\infty)^n,$$
     is bounded from $L^p((0,\infty)^n, \overset{n}{\underset{j=1}{\prod}} x_j^{2\lambda_j }dx)$ into itself,
     for every $1<p<\infty,$ and from $L^1((0,\infty)^n,$ $ \overset{n}{\underset{j=1}{\prod}} x_j^{2\lambda_j }dx)$ into
     $L^{1,\infty}((0,\infty)^n, \overset{n}{\underset{j=1}{\prod}} x_j^{2\lambda_j }dx).$
   \end{Prop}

   \begin{proof}
   In order to established the $L^p$-boundedness properties for the maximal operator $T_{\lambda_1,\ldots,\lambda_n}^{m,*}$ we consider the operator
   $$ T^{m,*}_{loc,\lambda_1,\ldots,\lambda_n}(f)(x)= \sup_{\varepsilon>0} \left| \int_{L(x), |x-y|>\varepsilon} f(y) \prod_{j=1}^n (x_jy_j)^{-\lambda_j}H^\phi(x,y) \prod_{j=1}^n y_j^{2\lambda_j}dy\right|, \ x\in (0,\infty)^n,$$
   where, for every $x=(x_1,\ldots,x_n)\in (0,\infty)^n,$
   $$ L(x)=\{y= (y_1,\ldots, y_n)\in(0,\infty)^n\ :\ x_j/2<y_j<2x_j, \ j=1,\ldots,n\}, $$
   and
   $$ H^\phi(x,y)= \int_0^\infty \phi(t) \frac{\partial}{\partial t} \left[\frac{e^{-|x-y|^2/4t}}{(4\pi t)^{n/2}} \right]dt, \quad x,y\in(0,\infty)^n.$$
   We have that
   \begin{align}\label{2.22}
      T_{\lambda_1,\ldots, \lambda_n}^{m,*} (f)(x) \leq &\left|T_{\lambda_1,\ldots, \lambda_n}^{m,*} (f)(x)-T_{loc,\lambda_1,\ldots, \lambda_n}^{m,*} (f)(x)\right| + T_{loc,\lambda_1,\ldots, \lambda_n}^{m,*} (f)(x)\nonumber\\
      \leq & \sup_{\varepsilon>0}\left| \int_{(0,\infty)^n\setminus L(x), |x-y|>\varepsilon} f(y) K^\phi_{\lambda_1,\ldots,\lambda_n}(x,y) \prod_{j=1}^n y_j^{2\lambda_j} dy\right|\nonumber\\
           & +  \sup_{\varepsilon>0}\left| \int_{L(x), |x-y|>\varepsilon} f(y) \left( K^\phi_{\lambda_1,\ldots,\lambda_n}(x,y)-\prod_{j=1}^n(x_jy_j)^{-\lambda_j}H^\phi(x,y)\right) \prod_{j=1}^n y_j^{2\lambda_j} dy\right|\nonumber\\
           & + T_{loc,\lambda_1,\ldots,\lambda_n}^{m,*}(f)(x)\nonumber\\
      \leq &  \int_{(0,\infty)^n\setminus L(x)} |f(y)|| K^\phi_{\lambda_1,\ldots,\lambda_n}(x,y)| \prod_{j=1}^n y_j^{2\lambda_j} dy\nonumber\\
           & + \int_{L(x)} |f(y)| \left| K^\phi_{\lambda_1,\ldots,\lambda_n}(x,y)-\prod_{j=1}^n(x_jy_j)^{-\lambda_j}H^\phi(x,y)\right| \prod_{j=1}^n y_j^{2\lambda_j} dy\nonumber\\
           & + T_{loc,\lambda_1,\ldots,\lambda_n}^{m,*}(f)(x)\nonumber\\
         = & \mathcal{G}_{\lambda_1,\ldots,\lambda_n}^\phi (|f|)(x) + \mathcal{L}_{\lambda_1,\ldots,\lambda_n}(|f|)(x) + T_{loc,\lambda_1,\ldots,\lambda_n}^{m,*}(f)(x), \quad x\in(0,\infty)^n.
   \end{align}
   We are going to show the $L^p$-boundedness properties for the operator $\mathcal{G}_{\lambda_1,\ldots,\lambda_n}^\phi,$ $\mathcal{L}_{\lambda_1,\ldots,\lambda_n}$ and $T_{loc,\lambda_1,\ldots,\lambda_n}^{m,*}.$\\

   We begin studying the operator $T_{loc,\lambda_1,\ldots, \lambda_n}^{m,*}.$ For every $j=(j_1,\ldots,j_n)\in \mathbb{Z}^n, $ the dyadic cube $Q_j$ is defined by
   $$ Q_j=\{ y=(y_1,\ldots,y_n)\in(0,\infty)^n :\ 2^{j_i} \leq y_i < 2^{j_i+1}, \ i=1,\ldots, n \},$$
   and the cube $\tilde{Q_j}$ is given by
   $$ \tilde{Q_j}=\{ y=(y_1,\ldots,y_n)\in(0,\infty)^n :\ 2^{j_i-1} \leq y_i < 2^{j_i+2}, \ i=1,\ldots, n \}.$$
   It is clear that if $j\in \mathbb{Z}^n$, $x\in Q_j$ and $y\in L(x),$ then $y\in \tilde{Q_j}.$ We can write
   \begin{align}\label{2.23}
      \int_{L(x), |x-y|>\varepsilon} f(y) &\prod_{l=1}^n(x_ly_l)^{-\lambda_l}H^\phi(x,y) \prod_{l=1}^n y_l^{2\lambda_l} dy=  \int_{\tilde{Q_j}, |x-y|>\varepsilon} f(y) \prod_{l=1}^n \left( \frac{y_l}{x_l}\right)^{\lambda_l}H^\phi(x,y) dy\nonumber\\
      &- \int_{\tilde{Q_j}\setminus L(x), |x-y|>\varepsilon} f(y) \prod_{l=1}^n \left( \frac{y_l}{x_l}\right)^{\lambda_l}H^\phi(x,y) dy, \ x\in Q_j, j\in \mathbb{Z}^n \text{ and } \varepsilon >0.
   \end{align}
   Let $j=(j_1,\ldots,j_n) \in \mathbb{Z}^n.$ It has $\tilde{Q_j}\setminus L(x) = \overset{n}{\underset{i=1}{\bigcup}}(\tilde{Q}_{j,i}^+\cup Q_{j,i}^-)$ where
   $$\tilde{Q}_{j,i}^+ = \{ y=(y_1,\ldots, y_n) \in (0,\infty)^n :\ 2^{j_l-1} \leq y_l <2^{j_l+2}, l=1,\ldots, n; l\neq i; \ 2x_i <y_i<2^{j_i+2} \}$$
   and
   $$\tilde{Q}_{j,i}^- = \{ y=(y_1,\ldots, y_n) \in (0,\infty)^n :\ 2^{j_l-1} \leq y_l <2^{j_l+2}, l=1,\ldots, n; l\neq i; \  2^{j_i-1}<y_i< x_i/2\}$$
   for $i=1,\ldots,n.$\\

   We have that
   \begin{align}\label{2.24}
      |H^\phi(x,y)|\leq &\int_0^\infty |\phi(t)| \left| \frac{\partial}{\partial t} \left( \frac{e^{-|x-y|^2/4t}}{t^{n/2}}\right)\right|dt \leq C \int_0^\infty \frac{e^{-|x-y|^2/8t}}{t^{n/2+1}}dt \nonumber\\
      \leq & C\frac{1}{|x-y|^n} \int_0^\infty \frac{e^{-1/u}}{u^{n/2+1}}du \leq C \frac{1}{|x-y|^n}, \quad x,y\in(0,\infty)^n.
   \end{align}

   By \eqref{2.24}, for every $\varepsilon>0,$ we get
   \begin{align*}
      &\left| \int_{\tilde{Q_j}\setminus L(x), |x-y|>\varepsilon} f(y) \prod_{l=1}^n \left( \frac{y_l}{x_l}\right)^{\lambda_l}H^\phi(x,y) dy \right| \leq \int_{\tilde{Q_j}\setminus L(x)} |f(y)| \prod_{l=1}^n \left( \frac{y_l}{x_l}\right)^{\lambda_l}|H^\phi(x,y)| dy\\
      \leq& \sum_{i=1}^n \left( \int_{\tilde{Q}_{j,i}^+} |f(y) |\prod_{l=1}^n \left( \frac{y_l}{x_l}\right)^{\lambda_l}|H^\phi(x,y) | dy+ \int_{\tilde{Q}_{j,i}^-} |f(y)| \prod_{l=1}^n \left( \frac{y_l}{x_l}\right)^{\lambda_l}|H^\phi(x,y)| dy\right)\\
      \leq & C \sum_{i=1}^n \int_{\tilde{Q}_{j,i}^+ \cup \tilde{Q}_{j,i}^-} \frac{|f(y)|}{(x_i^2+ \overset{n}{\underset{l=1,l\neq i}{\sum}}(x_l-y_l)^2)^{n/2}} dy\\
      \leq& C \sum_{i=1}^n \int_{\tilde{Q}_{j,i}^+ \cup \tilde{Q}_{j,i}^-} \frac{|f(y)|}{(2^{2j_i}+ \overset{n}{\underset{l=1,l\neq i}{\sum}}(x_l-y_l)^2)^{n/2}} dy, \ x\in Q_j.
   \end{align*}
   Then, for each $x\in Q_j,$
   \begin{equation}\label{2.25}
   \sup_{\varepsilon>0}\left|\int_{\tiny{\begin{array}{c}\tilde{Q_j}\setminus L(x)\\ |x-y|>\varepsilon\end{array}}} f(y) \prod_{l=1}^n \left( \frac{y_l}{x_l}\right)^{\lambda_l}H^\phi(x,y) dy \right| \leq C \sum_{i=1}^n \int_{\tilde{Q}_{j,i}^+ \cup \tilde{Q}_{j,i}^-} \frac{|f(y)|}{(2^{2j_i}+ \overset{n}{\underset{\tiny{\begin{array}{c}l=1\\l\neq i\end{array}}}{\sum}}(x_l-y_l)^2)^{n/2}} dy.
   \end{equation}
   For every $i=1,\ldots,n,$ we define
   $$ f_{j,i}(\bar{y}_i)= \int_{2^{j_i-1}}^{2^{j_i+2}}|f(y)|dy_i \chi_{\overset{n}{\underset{\tiny{l=1,l\neq i}}{\prod}} [2^{j_l-1}, 2^{j_l+2}]} (\bar{y}_i), \ \bar{y}_i = (y_1,\ldots, y_{i-1},y_{i+1},\ldots , y_n)\in(0,\infty)^{n-1}.$$
   From \eqref{2.25} it follows that
   \begin{align*}
      \sup_{\varepsilon>0}&\left|\int_{\tiny{\begin{array}{c}\tilde{Q_j}\setminus L(x)\\ |x-y|>\varepsilon\end{array}}} f(y) \prod_{l=1}^n \left( \frac{y_l}{x_l}\right)^{\lambda_l}H^\phi(x,y) dy \right| \leq C \sum_{i=1}^n \int_{(0,\infty)^{n-1}} \frac{f_{j,i}(\bar{y}_i)}{(2^{2j_i}+ |\bar{x}_i-\bar{y}_i|^2)^{n/2}} d\bar{y}_i\\
      \leq & C\sum_{i=1}^n \left( \sum_{k=0}^\infty \int_{\tiny{\begin{array}{l} (0,\infty)^{n-1}\\ 2^{k+j_i}<|\bar{y}_i-\bar{x}_i| < 2^{k+j_i+1}\end{array}}} \frac{f_{j,i}(\bar{y}_i)}{(2^{2j_i}+ |\bar{x}_i-\bar{y}_i|^2)^{n/2}} d\bar{y}_i \right.\\
      & + \left.\int_{\tiny{\begin{array}{l}(0,\infty)^{n-1}\\ |\bar{y}_i- \bar{x}_i| < 2^{j_i}\end{array}}} \frac{f_{j,i}(\bar{y}_i)}{(2^{2j_i}+ |\bar{x}_i-\bar{y}_i|^2)^{n/2}} d\bar{y}_i\right)\\
      \leq & C\sum_{i=1}^n \left( \sum_{k=0}^\infty \frac{1}{2^{nj_i}(1+2^{2k})^{n/2}} \int_{|\bar{y}_i-\bar{x}_i|< 2^{k+j_i+1}} f_{j,i}(\bar{y}_i) d\bar{y}_i + \frac{1}{2^{nj_i}} \int_{|\bar{y}_i-\bar{x}_i|<2^{j_i}} f_{j,i}(\bar{y}_i) d\bar{y}_i \right)\\
      \leq & C\sum_{i=1}^n \frac{1}{2^{j_i}}M_{n-1}(f_{j,i})(\bar{x}_i), \quad x\in Q_j.
   \end{align*}
   Here, $M_{n-1}$ represents the Hardy-Littlewood maximal function on $\mathbb{R}^{n-1}.$ We recall that $M_{n-1}$ is a bounded operator from $L^p(\mathbb{R}^{n-1})$ into itself, for every $1<p<\infty,$ and from $L^1(\mathbb{R}^{n-1})$ into $L^{1,\infty}(\mathbb{R}^{n-1}).$\\

   On the other hand, it is known that the maximal operator $T^{m,*}$ defined by
   $$ T^{m,*} (f)(x)= \sup_{\varepsilon>0} \left| \int_{|x-y|>\varepsilon}f(y) H^\phi(x,y) dy\right|, \quad x\in\mathbb{R}^n,$$
   is bounded from $L^p(\mathbb{R}^n,dx)$ into itself, for every $1<p<\infty,$ and from $L^1(\mathbb{R}^n,dx)$ into $L^{1,\infty}(\mathbb{R}^n,dx).$\\

   We denotes by $m_{\lambda_1,\ldots,\lambda_n}^{(n)}$ the positive measure given by $\displaystyle m_{\lambda_1,\ldots,\lambda_n}^{(n)}(E)=\int_E \prod_{j=1}^n x_j^{2\lambda_j}dx,$ for every Lebesgue measurable set $E.$\\

   Let $\gamma>0.$ We have that,
   \begin{align*}
      m&_{\lambda_1,\ldots,\lambda_n}^{(n)} (\{ x\in(0,\infty)^n: T_{loc,\lambda_1,\ldots,\lambda_n}^{m,*}(f)(x) > \gamma\})\\
      =& \sum_{j\in \mathbb{Z}^n} m_{\lambda_1,\ldots,\lambda_n}^{(n)} (\{ x\in Q_j: T_{loc,\lambda_1,\ldots,\lambda_n}^{m,*}(f)(x) > \gamma\})\\
      \leq & \sum_{j\in \mathbb{Z}^n} \Big(m_{\lambda_1,\ldots,\lambda_n}^{(n)} (\{ x\in Q_j : T^{m,*}(\prod_{l=1}^n(y_l/x_l)^{\lambda_l}f\chi_{\tilde{Q}_j})(x) > C\gamma\})\\
      & + m_{\lambda_1,\ldots,\lambda_n}^{(n)} \Big(\Big\{ x\in Q_j : \sup_{\varepsilon>0} \Big| \int_{\tilde{Q}_j\setminus L(x)}f(y) \prod_{l=1}^n \left( \frac{y_l}{x_l}\right)^{\lambda_l}H^\phi(x,y) dy \Big| > C\gamma\Big\}\Big)\Big)\\
      \leq & \sum_{j\in \mathbb{Z}^n} \left( \frac{C}{\gamma} \prod_{i=1}^n 2^{2j_i\lambda_i}\| f\chi_{\tilde{Q}_j}\|_{L^1(\mathbb{R}^n)}+ \sum_{i=1}^n m_{\lambda_1,\ldots, \lambda_n}^{(n)} (\{ x\in \tilde{Q}_j: M_{n-1}(f_{j,i})(\bar{x}_i)>2^{j_i}C\gamma\}\right)\\
      \leq & \sum_{j\in \mathbb{Z}^n}\left( \frac{C}{\gamma} \| f\chi_{\tilde{Q}_j}\|_{L^1((0,\infty)^n, \overset{n}{\underset{i=1}{\prod}} x_i^{2\lambda_i}dx)}\right.\\
      &+ \sum_{i=1}^n 2^{(2\lambda_i+1)j_i}m_{\lambda_1,\ldots,\lambda_{i-1},\lambda_{i+1},\ldots,\lambda_n}^{(n-1)}(\{ \bar{x}_i\in \prod_{l=1,\,l\neq i}^n [2^{j_l}, 2^{j_l+1}):
       M_{n-1}(f_{j,i})(\bar{x}_i)>2^{j_i}C\gamma \})\Big)\\
      \leq & \frac{C}{\gamma} \sum_{j\in \mathbb{Z}^n} \left( \| f\chi_{\tilde{Q}_j}\|_{L^1((0,\infty)^n, \overset{n}{\underset{i=1}{\prod}} x_i^{2\lambda_i}dx)}+ \sum_{i=1}^n \|f_{j,i}\|_{L^1(\mathbb{R}^{n-1})} \prod_{l=1}^n 2^{2j_l\lambda_l}\right)\\
      \leq & \frac{C}{\gamma} \sum_{j\in \mathbb{Z}^n}\| f\chi_{\tilde{Q}_j}\|_{L^1((0,\infty)^n, \overset{n}{\underset{i=1}{\prod}} x_i^{2\lambda_i}dx)} \leq \frac{C}{\gamma}\| f\|_{L^1((0,\infty)^n, \overset{n}{\underset{i=1}{\prod}} x_i^{2\lambda_i}dx)}, \quad f\in L^1((0,\infty)^n, \overset{n}{\underset{i=1}{\prod}} x_i^{2\lambda_i}dx).
   \end{align*}
    Hence, $T_{loc,\lambda_1,\ldots,\lambda_n}^{m,*}$ is bounded from $L^1((0,\infty)^n, \overset{n}{\underset{j=1}{\prod}} x_j^{2\lambda_j}dx)$ into $L^{1,\infty}((0,\infty)^n, \overset{n}{\underset{i=1}{\prod}} x_i^{2\lambda_i}dx).$ Also, if $1<p<\infty,$ we have
    \begin{align*}
       \| T&_{loc,\lambda_1,\ldots,\lambda_n}^{m,*}(f)\|^p_{L^p((0,\infty)^n, \overset{n}{\underset{j=1}{\prod}} x_j^{2\lambda_j}dx)} \leq \sum_{j\in \mathbb{Z}^n} \|T_{loc,\lambda_1,\ldots,\lambda_n}^{m,*}(f)\chi_{{Q}_j}\|^p_{L^p(\mathbb{R}^n,dx)} \prod_{l=1}^n 2^{2\lambda_lj_l}\\
       \leq& C \sum_{j\in \mathbb{Z}^n}\left( \int_{Q_j}|T^{m,*}(\prod_{l=1}^n(y_l/x_l)^{\lambda_l}f\chi_{\tilde{Q}_j})(x)|^p dx+ \sum_{i=1}^n \frac{1}{2^{j_ip}}\int_{Q_j} |M_{n-1}(f_{j,i})(\bar{x}_i)|^p d\bar{x}_idx_i\right)\prod_{l=1}^n 2^{2\lambda_lj_l}\\
       \leq & C \sum_{j\in \mathbb{Z}^n} \left(\int_{\tilde{Q}_j} |f(x)|^p dx + \sum_{i=1}^n \frac{1}{2^{j_ip}} \int_{2^{j_i}}^{2^{j_i+1}} \|f_{j,i} \|^p_{L^p(\mathbb{R}^{n-1})} dx_i \right)\prod_{l=1}^n 2^{2\lambda_lj_l}\\
       \leq & C \sum_{j\in \mathbb{Z}^n}\int_{\tilde{Q}_j} |f(x)|^p dx \prod_{l=1}^n 2^{2\lambda_lj_l} \leq C \|f\|^p_{L^p((0,\infty)^n, \overset{n}{\underset{j=1}{\prod}} x_j^{2\lambda_j}dx)}, \ f\in L^p((0,\infty)^n, \overset{n}{\underset{j=1}{\prod}} x_j^{2\lambda_j}dx).
    \end{align*}
     It is proved that $T_{loc,\lambda_1,\ldots,\lambda_n}^{m,*}$ is bounded from $L^p((0,\infty)^n, \overset{n}{\underset{j=1}{\prod}} x_j^{2\lambda_j}dx)$ into itself, for every $1<p<\infty.$\\

    We now analyze the operator $\mathcal{L}_{\lambda_1,\ldots,\lambda_n}.$ Let $\alpha>-1/2$.
         From \eqref{I2} and \eqref{Z0} we deduce that, if $t,u,v\in(0,\infty)$ and $uv/t> 1,$
    \begin{align}\label{2.27}
       \frac{\partial}{\partial t}W_t^\alpha (u,v) = &\left\{- \frac{\alpha+1/2}{t} \left( \frac{uv}{2t}\right)^{1/2}I_{\alpha-1/2}\left( \frac{uv}{2t}\right)- \frac{uv}{2t^2}\left( \frac{uv}{2t}\right)^{1/2}I_{\alpha+1/2}\left( \frac{uv}{2t}\right)\nonumber \right.\\
       &+ \left.\frac{u^2+v^2}{4t^2}\left( \frac{uv}{2t}\right)^{1/2}I_{\alpha-1/2}\left( \frac{uv}{2t}\right)\right\}\frac{(uv)^{-\alpha}}{\sqrt{2t}}e^{-(u^2+v^2)/4t}\nonumber\\
       =& \left\{-\frac{\alpha+1/2}{t} \left(1+ \mathcal{O}\left( \frac{t}{uv}\right)\right) - \frac{uv}{2t^2} \left( 1-\frac{4(\alpha+1/2)^2-1}{4}\frac{t}{uv}+ \mathcal{O}\left( \left( \frac{t}{uv}\right)^2\right) \right) \right.\nonumber\\
       & + \left.\frac{u^2+v^2}{4t^2}\left( 1-\frac{4(\alpha+1/2)^2-1}{4}\frac{t}{uv}+ \mathcal{O}\left( \left( \frac{t}{uv}\right)^2\right) \right)\right\}\frac{(uv)^{-\alpha}}{\sqrt{4\pi t}} e^{-(u-v)^2/4t}\nonumber \\
       =& \left\{\frac{(u-v)^2}{4t^2}-\frac{1}{2t}+ \frac{(u-v)^2}{4tuv}+\mathcal{O}\left( \frac{1}{uv}\right)+ \mathcal{O}\left( \frac{(u-v)^2}{(uv)^2}\right) \right\}\frac{(uv)^{-\alpha}}{\sqrt{4\pi t}} e^{-(u-v)^2/4t}\nonumber\\
       =& \frac{\partial}{\partial t} \left(\frac{e^{(u-v)^2/4t}}{\sqrt{4\pi t}}\right)(uv)^{-\alpha}+ \mathcal{O}\left( \frac{(uv)^{-\alpha}}{t^{1/2}}e^{-(u-v)^2/8t}\left( \frac{1}{uv}+\frac{(u-v)^2}{(uv)^2}\right)\right).
    \end{align}
    By proceeding in a similar way we can see that
    \begin{equation}\label{2.30}
       \left| W_t^\alpha(u,v)-(uv)^{-\alpha}\frac{e^{-(u-v)^2/4t}}{\sqrt{4\pi t}}\right|\leq C (uv)^{-\alpha-1}\sqrt{t}e^{-(u-v)^2/4t},\,\,\, t,u,v\in (0,\infty) \text{ and } \frac{uv}{t}> 1.
    \end{equation}
    We have that
    \begin{align*}
       \big| \mathcal{L}&_{\lambda_1,\ldots,\lambda_n}(f)(x)\big|\leq \int_{L(x)} |f(y)| \big| K_{\lambda_1,\ldots, \lambda_n}^\phi(x,y)-\prod_{j=1}^n (x_jy_j)^{-\lambda_j} H^\phi(x,y) \big| \prod_{j=1}^n y_j^{2\lambda_j} dy\\
       \leq & \int_{L(x)}|f(y)| \int_0^\infty |\phi(t)| \left| \frac{\partial}{\partial t}\left( \prod_{j=1}^n W_t^{\lambda_j}(x_j,y_j)\right)-\prod_{j=1}^n (x_jy_j)^{-\lambda_j}\frac{\partial}{\partial t} \left( \prod_{j=1}^n \frac{e^{-(x_j-y_j)^2/4t}}{\sqrt{4\pi t}}\right)\right| dt \prod_{j=1}^n y_j^{2\lambda_j} dy\\
       \leq & C \sum_{i=1}^n \int_{L(x)} |f(y)|\int_0^\infty \left|\prod_{\tiny{\begin{array}{c}j=1\\ j\neq i\end{array}}}^n W_t^{\lambda_j}(x_j,y_j) \frac{\partial}{\partial t} W_t^{\lambda_i}(x_i,y_i) \right.\\
       & - \prod_{j=1}^n (x_jy_j)^{-\lambda_j} \left.\prod_{\tiny{\begin{array}{c}j=1\\ j\neq i\end{array}}}^n \frac{e^{-(x_j-y_j)^2/4t}}{\sqrt{4\pi t}} \frac{\partial}{\partial t}\left( \frac{e^{-(x_i-y_i)^2/4t}}{\sqrt{4\pi t}}\right) \right|dt \prod_{j=1}^n y_j^{2\lambda_j} dy, \quad x\in(0,\infty)^n.
    \end{align*}
   We analyze the operator defined by the first summand. The other ones can be studied similarly. It follows that, for each $x \in (0,\infty)^n,$
   \begin{align}\label{2.31}
      \mathcal{L}^1_{\lambda_1,\ldots,\lambda_n}(f)(x) = & \int_{L(x)} |f(y)|\int_0^\infty \left|\prod_{j=2}^n W_t^{\lambda_j}(x_j,y_j) \frac{\partial}{\partial t} W_t^{\lambda_1}(x_1,y_1)\right. \nonumber\\
      &- \left. \prod_{j=1}^n (x_jy_j)^{-\lambda_j} \prod_{j=2}^n \frac{e^{-\frac{(x_j-y_j)^2}{4t}}}{\sqrt{4\pi t}} \frac{\partial}{\partial t}\left( \frac{e^{-\frac{(x_1-y_1)^2}{4t}}}{\sqrt{4\pi t}}\right) \right|dt \prod_{j=1}^n y_j^{2\lambda_j} dy\nonumber \\
      \leq & \int_{L(x)} |f(y)|\int_0^\infty \left| \frac{\partial}{\partial t} W_t^{\lambda_1}(x_1,y_1)-(x_1y_1)^{-\lambda_1}\frac{\partial}{\partial t} \left( \frac{e^{-\frac{(x_1-y_1)^2}{4t}}}{\sqrt{4 \pi t}} \right) \right| \prod_{j=2}^n W_t^{\lambda_j}(x_j,y_j)dt\nonumber\\
      &\times  \prod_{j=1}^n y_j^{2\lambda_j} dy  +  \sum_{i=2}^n \int_{L(x)} |f(y)|\int_0^\infty (x_1y_1)^{-\lambda_1}\left|\frac{\partial}{\partial t} \left( \frac{e^{-\frac{(x_1-y_1)^2}{4t}}}{\sqrt{4 \pi t}} \right) \right| \prod_{j=2}^{i-1}\left( \frac{e^{-\frac{(x_j-y_j)^2}{4t}}}{\sqrt{4\pi t}}\right) \nonumber\\
      & \times (x_jy_j)^{-\lambda_j}\left|W_t^{\lambda_i}(x_i,y_i)-(x_iy_i)^{-\lambda_i}\frac{e^{-\frac{(x_i-y_i)^2}{4t}}}{\sqrt{4\pi t}} \right| \prod_{j=i+1}^n W_t^{\lambda_j} (x_j,y_j) dt \prod_{j=1}^n y_j^{2\lambda_j} dy .
   \end{align}
    We now use \eqref{I1}, \eqref{I2}, \eqref{Z0}, and \eqref{2.27} to obtain the following
    \begin{align}\label{A}
       \mathcal{L}^{1,1}_{\lambda_1,\ldots,\lambda_n}&(f)(x) = \int_{L(x)} |f(y)|\int_0^\infty \left| \frac{\partial}{\partial t} W_t^{\lambda_1}(x_1,y_1)-(x_1y_1)^{-\lambda_1}\frac{\partial}{\partial t} \left( \frac{e^{-\frac{(x_1-y_1)^2}{4t}}}{\sqrt{4 \pi t}} \right) \right| \nonumber\\
      &\times \prod_{j=2}^n W_t^{\lambda_j}(x_j,y_j)dt \prod_{j=1}^n y_j^{2\lambda_j} dy \nonumber\\
      \leq & \int_{L(x)} |f(y)| \left\{ \int_0^{x_1y_1} \frac{(x_1y_1)^{-\lambda_1-1}}{t^{1/2}}e^{-\frac{(x_1-y_1)^2}{10t}}\prod_{j=2}^n \left( \frac{1}{x_j^{2\lambda_j+1}}+ (x_jy_j)^{-\lambda_j}\frac{e^{-\frac{(x_j-y_j)^2}{4t}}}{\sqrt{t}} \right)dt\right.\nonumber \\
      & + \left. \int_{x_1y_1}^\infty \left( \frac{1}{t^{\lambda_1+\frac{3}{2}}} + (x_1y_1)^{-\lambda_1}\frac{1}{t^{\frac{3}{2}}}\right)\prod_{j=2}^n \left( \frac{1}{x_j^{2\lambda_j+1}}+ (x_jy_j)^{-\lambda_j}\frac{e^{-\frac{(x_j-y_j)^2}{4t}}}{\sqrt{t}} \right)dt\right\} \prod_{j=1}^n y_j^{2\lambda_j} dy
    \end{align}
    for every $x\in (0,\infty)^n.$\\

    Then, the operator $\mathcal{L}^{1,1}_{\lambda_1,\ldots,\lambda_n}$ is controlled by operators of the following type:
    $$ \Lambda^l_{\lambda_1,\ldots,\lambda_n}(g)(x)= \sup_{t>0}\left|\int_{\frac{x_{l+1}}{2}}^{2x_{l+1}} \ldots \int_{\frac{x_n}{2}}^{2x_n} \prod_{j=l+1}^{n} x_j^{-2\lambda_j-1}\int_\frac{x_1}{2}^{2x_1}\ldots \int_\frac{x_l}{2}^{2x_l} \prod_{j=1}^l (x_jy_j)^{-\lambda_j} \frac{e^{-\frac{(x_j-y_j)^2}{10t}}}{\sqrt{t}}g(y) \prod_{j=1}^n y_j^{2\lambda_j}dy \right|,$$
    where $0\le l\le n$, $l\in \mathbb{N}.$ For each $0\le l\le n$, $l\in \mathbb{N},$ $\Lambda^l_{\lambda_1,\ldots,\lambda_n}$ is bounded operator from $L^p((0,\infty)^n, \overset{n}{\underset{j=1}{\prod}} x_j^{2\lambda_j}dx)$ into itself, for every $1<p<\infty,$ and from $L^1((0,\infty)^n, \overset{n}{\underset{j=1}{\prod}} x_j^{2\lambda_j}dx)$ into $L^{1,\infty}((0,\infty)^n, \overset{n}{\underset{j=1}{\prod}} x_j^{2\lambda_j}dx).$ Indeed, let $k\in \mathbb{N}.$ We consider the maximal operators
    $$\Omega_{\beta_1,\ldots,\beta_k} (g) (x) = \sup_{t>0}\left|\int_{\frac{x_{1}}{2}}^{2x_{1}} \ldots \int_{\frac{x_k}{2}}^{2x_k}\prod_{j=1}^k (x_jy_j)^{-\beta_j} \frac{e^{- \frac{(x_j-y_j)^2}{4t}}}{\sqrt{t}} g(y) \prod_{j=1}^k y_j^{2\beta_j}dy\right|, \quad x\in (0,\infty)^k, $$
    where $\beta_j \in \mathbb{R}, j=1,\ldots,k,$ and
    $$ M_k(g) (x)= \sup_{t>0}\left| \int_{(0,\infty)^k}\frac{e^{|x-y|^2/4t}}{t^{k/2}}g(y)dy\right|, \quad x\in (0,\infty)^k.$$
   For every $j= (j_1, \ldots, j_k)\in \mathbb{Z}^k $ we define
   $$ Q_j=\{ y=(y_1,\ldots,y_k)\in(0,\infty)^k :\ 2^{j_i} \leq y_i < 2^{j_i+1}, \ i=1,\ldots, k \},$$
   and
   $$ \tilde{Q_j}=\{ y=(y_1,\ldots,y_k)\in(0,\infty)^k :\ 2^{j_i-1} \leq y_i < 2^{j_i+2}, \ i=1,\ldots, k \}.$$

   Assume that $\gamma>0.$ Since, as it well known, the operator $M_k$ is bounded from $L^1((0,\infty)^k, dx)$ into $L^{1,\infty}((0,\infty)^k, dx),$ we have
   \begin{align*}
      m_{\beta_1,\ldots,\beta_k}^{(k)} (\{x&\in (0,\infty)^k : \Omega_{\beta_1,\ldots,\beta_k} (g)(x)>\lambda \}) = \sum_{j\in \mathbb{Z}^k} m_{\beta_1,\ldots,\beta_k}^{(k)}(\{ x\in Q_j : \Omega_{\beta_1,\cdot,\beta_k}(g)(x)>\lambda\})\\
      \leq &\sum_{j\in \mathbb{Z}^k}2^{2\overset{k}{\underset{i=1}{\sum}} \beta_i j_i}m_{0,\ldots,0}^{(k)}(\{ x\in \mathbb{R}^k: M_k(|g|\chi_{\tilde{Q}_j})(x) >C\gamma\})\\
      \leq &\frac{C}{\gamma} \sum_{j\in \mathbb{Z}^k}2^{2\overset{k}{\underset{i=1}{\sum}} \beta_i j_i} \int_{\tilde{Q}_j}|g(y)|dy \\
      \leq &\frac{C}{\gamma}  \int_{(0,\infty)^k}|g(y)|\prod_{j=1}^k y_j^{2\beta_j}dy, \quad g\in  L^1((0,\infty)^k, \overset{k}{\underset{j=1}{\prod}} x_j^{2\beta_j}dx).
   \end{align*}
   Also, $M_k$ is bounded from $L^p((0,\infty)^k, dx)$ into itself, for every $1<p<\infty,$ and it has
   \begin{align*}
      \int_{(0,\infty)^k}|\Omega_{\beta_1,\ldots,\beta_k}(g)(x)|^p &\prod_{j=1}^k x_j^{2\beta_j}dx  = 2^{2 \overset{k}{\underset{i=1}{\sum}} \beta_i j_i}\sum_{j\in \mathbb{Z}^k}\int_{Q_j}|\Omega_{\beta_1,\ldots,\beta_k}(g)(x)|^p dx \\
      \leq & C \sum_{j\in \mathbb{Z}^k} 2^{2 \overset{k}{\underset{i=1}{\sum}} \beta_i j_i}\int_{\mathbb{R}^k}|M_k(|g|\chi_{\tilde{Q}_j})(x)|^p dx \\
      \leq & C 2^{2 \overset{k}{\underset{i=1}{\sum}} \beta_i j_i}\sum_{j\in \mathbb{Z}^k} \int_{\tilde{Q}_j} |g(y)|^p dy \\
      \leq & C \int_{(0,\infty)^k} |g(y)|^p \prod_{j=1}^k y_j^{2\beta_j}dy, \quad g\in L^p((0,\infty)^k, \overset{k}{\underset{j=1}{\prod}}
      x_j^{2\beta_j}dx),
   \end{align*}
   and $\Omega_{\beta_1,\ldots,\beta_k}$ is bounded from $L^p((0,\infty)^k, \overset{k}{\underset{j=1}{\prod}} x_j^{2\beta_j}dx)$ into itself, for every $1<p<\infty.$\\

   On the other hand, it is not hard to see that, for every $k\in \mathbb{N}$ and $\beta_j\in \mathbb{R},$ $j=1,\ldots,k,$ the operator $Z_{\beta_1,\ldots,\beta_k}$ defined by
   $$Z_{\beta_1,\ldots,\beta_k}(g)(x)= \prod_{j=1}^k x_j^{-2\beta_j-1} \int_{x_1/2}^{2x_1} \ldots, \int_{x_k/2} ^{2x_k} g(y) \prod_{j=1}^k z_j^{2\beta_j}dy ,\quad x\in(0,\infty)^k, $$
   is bounded from $L^p((0,\infty)^k, \overset{k}{\underset{j=1}{\prod}} x_j^{2\beta_j}dx)$ into itself, for every $1\le p<\infty.$\\

   Then, the $L^p$-boundedness properties of operators $\Omega_{\beta_1,\ldots,\beta_k}$ and $Z_{\beta_1,\ldots,\beta_k},$ $\beta_j>-1/2$, $j=1,\ldots,k$, $k\in \mathbb{N},$ allow us, by using \cite[Proposition 1]{Din},
   to conclude that, for every $m\in \mathbb{N},$ the operator $\Lambda^m_{\lambda_1,\ldots,\lambda_n}$ is bounded from $L^p((0,\infty)^n, \overset{n}{\underset{j=1}{\prod}} x_j^{2\lambda_j}dx)$ into
   itself, for every $1<p<\infty,$ and from $L^1((0,\infty)^n, \overset{n}{\underset{j=1}{\prod}} x_j^{2\lambda_j}dx)$ into
   $L^{1,\infty}((0,\infty)^n, \overset{n}{\underset{j=1}{\prod}} x_j^{2\lambda_j}dx).$
   Hence, the operator $\mathcal{L}^{1,1}_{\lambda_1,\ldots,\lambda_n}$ is bounded from $L^p((0,\infty)^n, \overset{n}{\underset{j=1}{\prod}} x_j^{2\lambda_j}dx)$ into itself, for every
   $1<p<\infty,$ and from $L^1((0,\infty)^n, \overset{n}{\underset{j=1}{\prod}} x_j^{2\lambda_j}dx)$ into $L^{1,\infty}((0,\infty)^n, \overset{n}{\underset{j=1}{\prod}} x_j^{2\lambda_j}dx).$\\

   For every $i=2,\ldots,n,$ it has
   \begin{align}\label{Ai}
      \mathcal{L}^{1,i}_{\lambda_1,\ldots,\lambda_n}&(f)(x) = \int_{L(x)} |f(y)|\int_0^\infty (x_1y_1)^{-\lambda_1}\left|\frac{\partial}{\partial t} \left( \frac{e^{-\frac{(x_1-y_1)^2}{4t}}}{\sqrt{4 \pi t}} \right) \right| \prod_{j=2}^{i-1}\left( \frac{e^{-\frac{(x_j-y_j)^2}{4t}}}{\sqrt{4\pi t}}\right) \nonumber\\
      & \times (x_jy_j)^{-\lambda_j}\left|W_t^{\lambda_i}(x_i,y_i)-(x_iy_i)^{-\lambda_i}\frac{e^{-\frac{(x_i-y_i)^2}{4t}}}{\sqrt{4\pi t}} \right| \prod_{j=i+1}^n W_t^{\lambda_j} (x_j,y_j) dt \prod_{j=1}^n y_j^{2\lambda_j} dy\nonumber\\
      \leq & C \int_{L(x)} |f(y)| \left\{ \int_{0}^{x_iy_i} \frac{e^{-\overset{i-1}{\underset{j=1}{\sum}} \frac{(x_j-y_j)^2}{4t}}}{t^{\frac{i+1}{2}}} \prod_{j=1}^{i-1}(x_jy_j)^{-\lambda_j}(x_iy_i)^{-\lambda_i-1}\sqrt{t}e^{-\frac{(x_i-y_i)^2}{4t}}\right.\\
      & \times \prod_{j=i+1}^n \left( \frac{1}{x_j^{2\lambda_j + 1}}+(x_jy_j)^{-\lambda_j}\frac{e^{-\frac{(x_j-y_j)^2}{4t}}}{\sqrt{t}}\right)dt +  \int_{x_iy_i}^\infty \prod_{j=1}^{i-1}(x_jy_j)^{-\lambda_j}\frac{e^{-\overset{i-1}{\underset{j=1}{\sum}} \frac{(x_j-y_j)^2}{4t}}}{t^{\frac{i+1}{2}}}\left( \frac{1}{t^{\lambda_i + \frac{1}{2}}} \right.\nonumber\\
      & +\left.\left. (x_iy_i)^{-\lambda_i}\frac{e^{-\frac{(x_i-y_i)^2}{4t}}}{\sqrt{t}}\right)\prod_{j=i+1}^n \left( \frac{1}{x_j^{2\lambda_j+1}} + (x_jy_j)^{-\lambda_j}\frac{e^{-\frac{(x_j-y_j)^2}{4t}}}{\sqrt{t}} \right) dt \right\} \prod_{j=1}^n y_j^{2\lambda_j}dy,\,\,\,x\in (0,\infty)^n.\nonumber
   \end{align}
   Then, the operator $\mathcal{L}^{1,i}_{\lambda_1,\ldots,\lambda_n}$ can be controlled by operator of type $\Lambda^m_{\lambda_1,\ldots,\lambda_n}$, $m\in \mathbb{N}.$ Thus, we conclude that $ \mathcal{L}^{1,i}_{\lambda_1,\ldots,\lambda_n}$ is bounded from $L^p((0,\infty)^n, \overset{n}{\underset{j=1}{\prod}} x_j^{2\lambda_j}dx)$ into itself, for every $1<p<\infty,$ and from $L^1((0,\infty)^n, \overset{n}{\underset{j=1}{\prod}} x_j^{2\lambda_j}dx)$ into $L^{1,\infty}((0,\infty)^n, \overset{n}{\underset{j=1}{\prod}} x_j^{2\lambda_j}dx).$\\

   From the assertions proved in \eqref{A} and \eqref{Ai}, $i=2,\ldots,n,$ we deduce that the operator $\mathcal{L}^1_{\lambda_1,\ldots,\lambda_n}$ is bounded from $L^p((0,\infty)^n, \overset{n}{\underset{j=1}{\prod}} x_j^{2\lambda_j}dx)$ into itself, for every $1<p<\infty,$ and from $L^1((0,\infty)^n, \overset{n}{\underset{j=1}{\prod}}$ $ x_j^{2\lambda_j}dx)$ into $L^{1,\infty}((0,\infty)^n, \overset{n}{\underset{j=1}{\prod}} x_j^{2\lambda_j}dx).$\\

   Thus we conclude that the operator $\mathcal{L}_{\lambda_1,\ldots,\lambda_n}$ is bounded from $L^p((0,\infty)^n, \overset{n}{\underset{j=1}{\prod}} x_j^{2\lambda_j}dx)$ into itself, for every $1<p<\infty,$ and from $L^1((0,\infty)^n, \overset{n}{\underset{j=1}{\prod}} x_j^{2\lambda_j}dx)$ into $L^{1,\infty}((0,\infty)^n, \overset{n}{\underset{j=1}{\prod}} x_j^{2\lambda_j}dx).$\\

   Finally we study the operator $\mathcal{G}_{\lambda_1,\ldots,\lambda_n}^\phi.$ It is clear that $\mathcal{G}_{\lambda_1,\ldots,\lambda_n}^\phi$ can be decomposed in a sum of operators like the following one
   \begin{align}\label{2.31.2}
     S_{\lambda_1,\ldots,\lambda_n}^{k,l}(g)(x) =& \int_0^\frac{x_1}{2} \ldots \int_0^\frac{x_k}{2} \int_\frac{x_{k+1}}{2}^{2x_{k+1}}\ldots \int_\frac{x_{l}}{2}^{2x_{l}} \int_{2x_{l+1}}^\infty \ldots \int_{2x_{n}}^\infty g(y) K_{\lambda_1,\ldots,\lambda_n}^\phi (x,y) \prod_{j=1}^n y_j^{2\lambda_j}dy\nonumber \\
     =& \sum_{i=1}^n \int_0^\frac{x_1}{2} \ldots \int_0^\frac{x_k}{2} \int_\frac{x_{k+1}}{2}^{2x_{k+1}}\ldots \int_\frac{x_{l}}{2}^{2x_{l}} \int_{2x_{l+1}}^\infty \ldots \int_{2x_{n}}^\infty g(y) K_{\lambda_1,\ldots,\lambda_n}^{\phi,i} (x,y) \prod_{j=1}^n y_j^{2\lambda_j}dy\nonumber  \\
     =& \sum_{i=1}^n S_{\lambda_1,\ldots,\lambda_n}^{k,l,i}(g)(x),\quad x\in (0,\infty)^n,
   \end{align}
   where $0\leq k\leq l\leq n,$ $ k\neq 0$ or $l\neq n,$ and
   $$ K_{\lambda_1,\ldots,\lambda_n}^{\phi,i}(x,y)=\int_0^\infty\prod_{j=1}^{i-1}W_t^{\lambda_j}(x_j,y_j) \frac{\partial}{\partial t}W_t^{\lambda_i}(x_i,y_i) \prod_{j=i+1}^n W_t^{\lambda_j}(x_j,y_j)\phi(t)dt, \quad x,y \in (0,\infty)^n,$$
   for every $i=1,\cdot,n$.
   From \eqref{I2} we deduce that, if $\alpha>-1/2,$
   \begin{align}\label{2.32}
      \left| W_t^\alpha(u,v)\right|\leq & C \frac{(uv)^{-\alpha} e^{-v^2/16t}}{\sqrt{t}} \leq C(uv)^{-\alpha-1/2}\left( \frac{uv}{t}\right)^{1/2}e^{-v^2/16t} \nonumber\\
      \leq & Ct^{-\alpha-1/2}\frac{v}{\sqrt{t}}e^{-v^2/16t} \leq \frac{C}{t^{\alpha+1/2}}e^{-v^2/20t},\,\, t,u,v \in (0,\infty), \frac{uv}{t}>1 \text{ and } 2u<v<\infty.
   \end{align}
   By \eqref{I1} and \eqref{2.32} we conclude that for $\alpha>-1/2$
   \begin{equation}\label{2.33}
   \left| W_t^\alpha(u,v)\right| \leq \frac{C}{t^{\alpha+1/2}}e^{-v^2/20t}, \quad t,u,v \in (0,\infty), \text{ and } 2u<v<\infty.
   \end{equation}
   From \eqref{2.27} it follows, if $\alpha>-1/2,$
   \begin{align}\label{2.34}
      \left| \frac{\partial}{\partial t} W_t^\alpha(u,v) \right| \leq &C e^{-v^2/16t}\left(\frac{1} {t^{3/2}}(uv)^{-\alpha}+ \frac{(uv)^{-\alpha-1}}{t^{1/2}}\right)\nonumber\\
      \leq & C \frac{e^{-v^2/16t}}{t^{3/2}}(uv)^{-\alpha} \leq C \frac{e^{-v^2/20t}}{t^{\alpha+3/2}}, \quad t,u,v \in (0,\infty), \,\,\frac{uv}{t}>1 \text{ and } 2u<v<\infty.
   \end{align}
   By \eqref{Z0} and \eqref{2.34} we get for every $\alpha>-1/2$
   \begin{equation}\label{H4}
   \Big|\frac{\partial}{\partial t} W_t^\alpha(u,v) \Big| \leq C \frac{e^{-v^2/20t}}{t^{\alpha+3/2}},\,\,\,t,u,v\in (0,\infty),\,\,\,2u<v<\infty.
   \end{equation}

   According to \eqref{2.33} and \eqref{H4} we infer
   \begin{align*}
      \left| S_{\lambda_1,\ldots,\lambda_n}^{0,0,i}(g)(x)\right|\leq &C\int_{2x_1}^\infty \ldots \int_{2x_n}^\infty |g(y)| \int_0^\infty \frac{e^{-\overset{n}{\underset{j=1}{\sum}} y_j^2/20t}}{t^{\overset{n}{\underset{j=1}{\sum}} (\lambda_j +1/2)+1}} dt \prod_{j=1}^n y_j^{2\lambda_j}dy\\
      \leq & C \int_{2x_1}^\infty \ldots \int_{2x_n}^\infty |g(y)|
      \frac{1}{(\overset{n}{\underset{j=1}{\sum}} y_j^2)^{\overset{n}{\underset{j=1}{\sum}}(\lambda_j+1/2)}} \prod_{j=1}^n y_j^{2\lambda_j}dy\\
      \leq & C\int_{2x_1}^\infty \ldots \int_{2x_n}^\infty |g(y)|\frac{1}{y_1\ldots y_n} dy, \quad x\in (0,\infty)^n \text{ and } i=1,\ldots,n.
   \end{align*}
   It is not hard to see that, for every $k\in \mathbb{N}$ the operator
   $$\mathcal{S}_k(g)(x)= \int_{2x_1}^\infty \ldots \int_{2x_k}^\infty |g(y)| \frac{1}{y_1\ldots y_k} dy, \quad x\in(0,\infty)^k,$$
   is bounded from $L^p((0,\infty)^k, \overset{k}{\underset{j=1}{\prod}} x_j^{2\beta_j}dx)$ into itself, for every $1\leq p<\infty$ and $\beta_j>-1/2,$ $j=1,\ldots,k$.\\

   Hence the operator $S_{\lambda_1,\ldots,\lambda_n}^{0,0,i}$ is bounded from $L^p((0,\infty)^n, \overset{n}{\underset{j=1}{\prod}}$ $ x_j^{2\lambda_j}dx),$ for every $1\leq p < \infty$ and $i=1,\ldots,n.$\\

   Let $i,k\in \mathbb{N},$ $1\leq i\leq k,$ and $\beta_j>-1/2,$ $j=1,\ldots,k.$ We define the operator $ \mathcal{H}_{\beta_1,\ldots,\beta_k}$ by
   \begin{align*}
      \mathcal{H}_{\beta_1,\ldots,\beta_k}(g)(x)=& \int_0^\frac{x_1}{2} \ldots \int_0^\frac{x_k}{2} |g(y)| \int_0^\infty \prod_{j=1}^{i-1} W_t^{\beta_j} (x_j,y_j) \Big|\frac{\partial}{\partial t} W_t^{\beta_i}(x_i,y_i)\Big|\\
      &\times \prod_{j=i+1}^k W_t^{\beta_j}(x_j,y_j) dt \prod_{j=1}^k y_j^{2\beta_j} dy, \quad x\in (0,\infty)^k.
   \end{align*}
   By taking into account symmetries, \eqref{2.33}, and \eqref{H4}, we get that
   \begin{align}\label{2.36}
       |\mathcal{H}_{\beta_1,\ldots,\beta_k}(g)(x)|\leq& C\int_0^\frac{x_1}{2} \ldots \int_0^\frac{x_k}{2} |g(y)| \int_0^\infty  \frac{e^{-\overset{k}{\underset{j=1}{\sum}} x_j^2/20t}}{t^{\overset{k}{\underset{j=1}{\sum}} (\beta_j +1/2)+1}} dt \prod_{j=1}^k y_j^{2\beta_j}dy \nonumber \\
       \leq & \frac{C}{(\overset{k}{\underset{j=1}{\sum}} x_j^2)^{\overset{k}{\underset{j=1}{\sum}}(\beta_j+1/2)}} \int_0^\frac{x_1}{2} \ldots \int_0^\frac{x_k}{2} |g(y)|  \prod_{j=1}^k
       y_j^{2\beta_j}dy.
   \end{align}
   The operator $H_{\beta_1,\ldots,\beta_k}$ given by
   $$ H_{\beta_1,\ldots,\beta_k}(g)(x) = \frac{1}{(\overset{k}{\underset{j=1}{\sum}} x_j^2)^{\overset{k}{\underset{j=1}{\sum}}(\beta_j+1/2)}} \int_0^\frac{x_1}{2} \ldots \int_0^\frac{x_k}{2} g(y)  \prod_{j=1}^k y_j^{2\beta_j}dy $$
   is bounded from $L^p((0,\infty)^k, \overset{k}{\underset{j=1}{\prod}} x_j^{2\beta_j}dx)$ into itself, for every $1<p<\infty,$ and from $L^1((0,\infty)^k,$ $ \overset{k}{\underset{j=1}{\prod}} x_j^{2\beta_j}dx)$ into $L^{1,\infty}((0,\infty)^k,$ $ \overset{k}{\underset{j=1}{\prod}} x_j^{2\beta_j}dx).$ Indeed, assume that $\gamma >0.$ We have
   \begin{align*}
      m_{\beta_1,\ldots,\beta_k}^{(k)}&(\{ x\in (0,\infty)^k :| H_{\beta_1,\ldots,\beta_k }(g)(x)| >\gamma \} )\\
      \leq & m_{\beta_1,\ldots,\beta_k}^{(k)}\Big(\Big\{ x\in (0,\infty)^k : \sum_{j=1}^k x_j^2 \leq \left( \frac{1}{\gamma} \|g\|_{L^1((0,\infty)^k, \overset{k}{\underset{j=1}{\prod}} x_j^{2\beta_j}dx)}   \right)^{\frac{1}{\overset{k}{\underset{j=1}{\sum}} (\beta_j +1/2)}} \Big\} \Big)\\
      \leq & m_{\beta_1,\ldots,\beta_k}^{(k)}\Big(\Big\{ x\in (0,\infty)^k : 0\leq x_j \leq  \left( \frac{1}{\gamma} \|g\|_{L^1((0,\infty)^k, \overset{k}{\underset{j=1}{\prod}} x_j^{2\beta_j}dx)}   \right)^{\frac{1}{\overset{k}{\underset{j=1}{\sum}} (2\beta_j +1)}},\,\,j=1,\dots,k \Big\}\Big)\\
      \leq & \frac{C}{\gamma} \|g\|_{L^1((0,\infty)^k, \overset{k}{\underset{j=1}{\prod}} x_j^{2\beta_j}dx)},\,\,\,g\in L^1((0,\infty)^k,\prod_{j=1}^kx_j^{2\beta_j}dx).
   \end{align*}
   Hence, $H_{\beta_1,\ldots,\beta_k}$ is bounded from $L^1((0,\infty)^k, \overset{k}{\underset{j=1}{\prod}} x_j^{2\beta_j}dx)$ into $L^{1,\infty}((0,\infty)^k, \overset{k}{\underset{j=1}{\prod}} x_j^{2\beta_j}dx).$\\

   On the other hand, we can write
   $$ H_{\beta_1,\ldots,\beta_k}(g)(x) \leq \frac{1}{x_1^{2\beta_1+1}} \int_0^{x_1}\frac{1}{x_2^{2\beta_2+1}} \int_0^{x_2}\ldots \frac{1}{x_k^{2\beta_k+1}} \int_0^{x_k}|g(y)|\prod_{j=1}^k y_j^{2\beta_j}dy, \quad x\in (0,\infty)^k. $$
   Since, as it is well known, the Hardy type operator $H_\beta$ given by, for every $\beta>-1/2,$
   $$ H_{\beta}(g)(x)= \frac{1}{x^{2\beta+1}} \int_0^{x}g(y)y^{2\beta}dy, \quad x\in (0,\infty),$$
   is bounded from $L^p((0,\infty),x^{2\beta}dx)$ into itself, for every $1<p<\infty$ (see \cite{Mu}), from
   \cite[Proposition 1]{Din} we deduce that $H_{\beta_1, \ldots, \beta_k}$ is a bounded operator from
   $L^p((0,\infty)^k, \overset{k}{\underset{j=1}{\prod}} x_j^{2\beta_j}dx)$ into itself, for every $1<p<\infty.$\\

   Then, from \eqref{2.36} we infer that the operator $\mathcal{H}_{\beta_1,\ldots,\beta_k}$ is bounded from $L^p((0,\infty)^k, \overset{k}{\underset{j=1}{\prod}} x_j^{2\beta_j}dx)$ into itself, for every $1<p<\infty,$ and from $L^1((0,\infty)^k, \overset{k}{\underset{j=1}{\prod}} x_j^{2\beta_j}dx)$ into $L^{1,\infty}((0,\infty)^k, \overset{k}{\underset{j=1}{\prod}} x_j^{2\beta_j}dx).$
   Also, $S_{\lambda_1,\ldots,\lambda_n}^{n,n,i}$ is bounded from $L^p((0,\infty)^n, \overset{n}{\underset{j=1}{\prod}} x_j^{2\lambda_j}dx)$ into itself, for every $1<p<\infty,$ and from $L^1((0,\infty)^n, \overset{n}{\underset{j=1}{\prod}} x_j^{2\lambda_j}dx)$ into $L^{1,\infty}((0,\infty)^n, \overset{n}{\underset{j=1}{\prod}} x_j^{2\lambda_j}dx),$ for every $i=1,\ldots,n.$\\

   Assume that $1\leq i\leq k < n.$ By using \eqref{2.33} and
   \eqref{H4} we have that
   \begin{align*}
      \left|S_{\lambda_1,\ldots,\lambda_n}^{k,k,i}(g)(x) \right|\leq & C \int_0^\frac{x_1}{2} \ldots \int_0^\frac{x_k}{2} \int_{2x_{k+1}}^\infty \ldots \int_{2x_{n}}^\infty \int_0^\infty \frac{e^{-(\overset{k}{\underset{j=1}{\sum}} x_j^2+\overset{n}{\underset{j=k+1}{\sum}} y_j^2 ) /20t}}{t^{\overset{n}{\underset{j=1}{\sum}} (\lambda_j +1/2)+1}} dt|g(y)| \prod_{j=1}^n y_j^{2\lambda_j}dy\\
      \leq & C \int_0^\frac{x_1}{2} \ldots \int_0^\frac{x_k}{2} \int_0^\infty \frac{e^{-\overset{k}{\underset{j=1}{\sum}} x_j^2 /20t}}{t^{\overset{k}{\underset{j=1}{\sum}} (\lambda_j +1/2)+1}} dt\int_{2x_{k+1}}^\infty \ldots \int_{2x_{n}}^\infty \frac{|g(y)|}{\overset{n}{\underset{j=k+1}{\prod}}y_j^{2\lambda_j+1}} \prod_{j=1}^n y_j^{2\lambda_j}dy\\
      \leq & C \frac{1}{\Big(\overset{k}{\underset{j=1}{\sum}} x_j^2\Big)^{\overset{k}{\underset{j=1}{\sum}}(\lambda_j+1/2)}} \int_0^\frac{x_1}{2} \ldots \int_0^\frac{x_k}{2}\left( \int_{2x_{k+1}}^\infty \ldots \int_{2x_{n}}^\infty  \frac{|g(y)|}{y_{k+1}\ldots y_n} dy_n\ldots dy_{k+1} \right)\\
      & \times  \prod_{j=1}^k y_j^{2\lambda_j}dy_k\ldots dy_1, \quad x\in (0,\infty)^n.
   \end{align*}
   According to \cite[Proposition 1]{Din} and by taken into account the $L^p$-boundedness properties of the operator $H_{\lambda_1, \ldots, \lambda_k}$ and $\mathcal{S}_{n-k}$ we conclude that the operator $S_{\lambda_1,\ldots,\lambda_n}^{k,k,i}$ is bounded from
   $L^p((0,\infty)^n, \overset{n}{\underset{j=1}{\prod}} x_j^{2\lambda_j}dx)$ into itself, for every $1<p<\infty,$ and from $L^1((0,\infty)^n, \overset{n}{\underset{j=1}{\prod}} x_j^{2\lambda_j}dx)$ into $L^{1,\infty}((0,\infty)^n, \overset{n}{\underset{j=1}{\prod}} x_j^{2\lambda_j}dx).$\\

   In a similar way we can see that if $0<k<i\leq n$ the operator $S_{\lambda_1,\ldots,\lambda_n}^{k,k,i}$ is bounded from $L^p((0,\infty)^n, \overset{n}{\underset{j=1}{\prod}} x_j^{2\lambda_j}dx)$ into itself, for every $1<p<\infty,$ and from $L^1((0,\infty)^n, \overset{n}{\underset{j=1}{\prod}} x_j^{2\lambda_j}dx)$ into $L^{1,\infty}((0,\infty)^n, \overset{n}{\underset{j=1}{\prod}} x_j^{2\lambda_j}dx).$\\

   Let $1\leq i \leq k< n.$ By using \eqref{I1}, \eqref{I2}, \eqref{2.33} and \eqref{H4}, we get
   \begin{align}\label{2.37}
      \left|S_{\lambda_1,\ldots,\lambda_n}^{k,n,i} (g)(x)\right| \leq &C \int_0^\frac{x_1}{2} \ldots \int_0^\frac{x_k}{2} \int_\frac{x_{k+1}}{2}^{2x_{k+1}}\ldots \int_\frac{x_{n}}{2}^{2x_{n}}\int_0^\infty \frac{e^{-\overset{k}{\underset{j=1}{\sum}} x_j^2 /20t}}{t^{\overset{k}{\underset{j=1}{\sum}} (\lambda_j +1/2)+1}} \nonumber\\
      & \times \prod_{j=k+1}^n \left( \frac{1}{x_j^{2\lambda_j+1}} + (x_jy_j)^{-\lambda_j}\frac{e^{-(x_j-y_j)^2/4t}}{\sqrt{t}}\right) dt |g(y)| \prod_{j=1}^n y_j^{2\lambda_j}dy, \quad x\in (0,\infty)^n.
   \end{align}
   Suppose that $k<l\leq n,$ $l\in \mathbb{N},$ and define
   \begin{align*}
      \mathbb{S}_{\lambda_1,\ldots,\lambda_n}^{k,l} (g)(x) =&\int_0^\frac{x_1}{2} \ldots \int_0^\frac{x_k}{2} \int_\frac{x_{k+1}}{2}^{2x_{k+1}}\ldots \int_\frac{x_{n}}{2}^{2x_{n}}\int_0^\infty \frac{e^{-(\overset{k}{\underset{j=1}{\sum}} x_j^2 + \overset{l}{\underset{j=k+1}{\sum}} (x_j-y_j)^2)/20t}}{t^{\overset{k}{\underset{j=1}{\sum}} (\lambda_j +1/2)+1+ (l-k)/2}}dt\nonumber\\
      & \times \prod_{j=k+1}^l (x_jy_j)^{-\lambda_j} \prod_{j=l+1}^n  \frac{1}{x_j^{2\lambda_j+1}}g(y) \prod_{j=1}^n y_j^{2\lambda_j}dy, \quad x\in (0,\infty)^n.
   \end{align*}
   We have that, for each  $ x\in (0,\infty)^n,$
   \begin{align}\label{2.38}
      \big|\mathbb{S}_{\lambda_1,\ldots,\lambda_n}^{k,l}& (g)(x)\big| \leq C\int_0^\frac{x_1}{2} \ldots \int_0^\frac{x_k}{2} \int_\frac{x_{k+1}}{2}^{2x_{k+1}}\ldots \int_\frac{x_{l}}{2}^{2x_{l}} \frac{\overset{l}{\underset{j=k+1}{\prod}} (x_jy_j)^{-\lambda_j}}{ (\overset{k}{\underset{j=1}{\sum}} x_j^2 + \overset{l}{\underset{j=k+1}{\sum}} (x_j-y_j)^2  )^{\overset{k}{\underset{j=1}{\sum}} (\lambda_j+1/2)+ (l-k)/2}}\nonumber\\
      & \times \left( \frac{1}{\overset{n}{\underset{j=l+1}{\prod}} x_j^{2\lambda_j+1}}\int_\frac{x_{l+1}}{2}^{2x_{l+1}}\ldots \int_\frac{x_{n}}{2}^{2x_{n}}|g(y)| \prod_{j=l+1}^n y_j^{2\lambda_j}dy_n\ldots dy_{l+1} \right) \prod_{j=1}^l y_j^{2\lambda_j}dy_l\ldots dy_1.
   \end{align}
   Assume that $r,s\in \mathbb{N},$ $0<s<r,$ and $\beta_j >-1/2,$ $j=1,\ldots, r.$ We consider the operator
    \begin{align*}
      Y_{\beta_1,\ldots,\beta_r}^{s} (g)(x) =&\int_0^\frac{x_1}{2} \ldots \int_0^\frac{x_s}{2} \int_\frac{x_{s+1}}{2}^{2x_{s+1}}\ldots \int_\frac{x_{r}}{2}^{2x_{r}} \frac{\overset{r}{\underset{j=s+1}{\prod}} (x_jy_j)^{-\beta_j}}{ (\overset{s}{\underset{j=1}{\sum}} x_j^2 + \overset{r}{\underset{j=s+1}{\sum}} (x_j-y_j)^2  )^{\overset{s}{\underset{j=1}{\sum}} (\lambda_j+1/2)+ (r-s)/2}}\\
      & \times g(y) \prod_{j=1}^r y_j^{2\lambda_j}dy, \quad x\in(0,\infty)^r.
   \end{align*}
   By proceeding as in the proof of Case 3 in \cite{NS1} we can see that the operator $Y_{\beta_1,\ldots,\beta_r}^s$ is bounded from
   $L^1((0,\infty)^r, \overset{r}{\underset{j=1}{\prod}} x_j^{2\beta_j}dx)$ into $L^{1,\infty}((0,\infty)^r, \overset{r}{\underset{j=1}{\prod}} x_j^{2\beta_j}dx).$
   Since the operator $Z_{\lambda_{l+1},\ldots,\lambda_n}$ is bounded from
   $L^1((0,\infty)^r, \overset{r}{\underset{j=1}{\prod}} x_j^{2\beta_j}dx)$ into itself, by \cite[Proposition 1]{Din},
   we deduce from \eqref{2.38} that the operator $\mathbb{S}_{\lambda_1,\ldots,\lambda_n}^{k,l}$ is bounded from
   $L^1((0,\infty)^n, \overset{n}{\underset{j=1}{\prod}} x_j^{2\lambda_j}dx)$ into $L^{1,\infty}((0,\infty)^n, \overset{n}{\underset{j=1}{\prod}} x_j^{2\lambda_j}dx).$\\

   On the other hand we also have that
   \begin{align*}
      \big|\mathbb{S}_{\lambda_1,\ldots,\lambda_n}^{k,l}& (g)(x)\big| \leq C\int_0^\frac{x_1}{2} \ldots \int_0^\frac{x_k}{2} \int_0^\infty \frac{e^{-\overset{k}{\underset{j=1}{\sum}} x_j^2 /20t}}{t^{\overset{k}{\underset{j=1}{\sum}} (\lambda_j +1/2)+1}} dt \left\{ \frac{1}{\overset{n}{\underset{j=l+1}{\prod}} x_j^{2\lambda_j+1}}\int_\frac{x_{l+1}}{2}^{2x_{l+1}}\ldots \int_\frac{x_{n}}{2}^{2x_{n}}\right. \\
      &  \left( \sup_{t>0} \int_\frac{x_{k+1}}{2}^{2x_{k+1}}\ldots \int_\frac{x_{l}}{2}^{2x_{l}} \prod_{j=k+1}^l (x_jy_j)^{-\lambda_j} \frac{e^{-(x_j-y_j)^2/4t}}{\sqrt{t}}|g(y)| \prod_{j=k+1}^l y_j^{2\lambda_j}d{y_l}\ldots d{y_{k+1}} \right)\\
      & \times \left. \prod_{j=l+1}^n y_j^{2\lambda_j} dy_n \ldots dy_{l+1}\right\} \prod_{j=1}^k y_j^{2\lambda_j} dy_1 \ldots dy_{k}\\
      \leq & C \frac{1}{\overset{k}{\underset{j=1}{\prod}} x_j^{2\lambda_j+1}}\int_0^\frac{x_1}{2} \ldots \int_0^\frac{x_k}{2} \left\{ \frac{1}{\overset{n}{\underset{j=l+1}{\prod}} x_j^{2\lambda_j+1}}\int_\frac{x_{l+1}}{2}^{2x_{l+1}}\ldots \int_\frac{x_{n}}{2}^{2x_{n}} \left( \sup_{t>0} \int_\frac{x_{k+1}}{2}^{2x_{k+1}}\ldots \int_\frac{x_{l}}{2}^{2x_{l}}\right. \right.\\
      & \times \left. \left.\prod_{j=k+1}^n (x_jy_j)^{-\lambda_j} \frac{e^{-(x_j-y_j)^2/4t}}{\sqrt{t}}|g(y)| \prod_{j=k+1}^l y_j^{2\lambda_j}d{y_l}\ldots d{y_{k+1}} \right) \prod_{j=l+1}^n y_j^{2\lambda_j} dy_n \ldots dy_{l+1}\right\}\\
      &\times \prod_{j=1}^k y_j^{2\lambda_j} dy_1 \ldots dy_{k}, \quad x\in (0,\infty)^n.\\
   \end{align*}
   According to the $L^p$-boundedness properties of the operators $\Omega_{\beta_1,\ldots,\beta_r}$, $Z_{\beta_1,\ldots,\beta_r}$,  and $H_{\beta_1,\ldots,\beta_r}$, $\beta_j>-1/2,$ $j=1,\ldots,r$, and $r\in \mathbb{N}$, by using
   \cite[Proposition 1]{Din}, we obtain that $\mathbb{S}_{\lambda_1,\ldots,\lambda_n}^{k,l}$ is bounded from
   $L^p((0,\infty)^n, \overset{n}{\underset{j=1}{\prod}} x_j^{2\lambda_j}dx)$ into itself, for every $1<p<\infty.$\\

   We conclude that $S_{\lambda_1,\ldots,\lambda_n}^{k,n,i}$ is bounded from $L^p((0,\infty)^n, \overset{n}{\underset{j=1}{\prod}} x_j^{2\lambda_j}dx)$ into itself, for every $1<p<\infty,$ and from $L^1((0,\infty)^n, \overset{n}{\underset{j=1}{\prod}} x_j^{2\lambda_j}dx)$ into $L^{1,\infty}((0,\infty)^n, \overset{n}{\underset{j=1}{\prod}} x_j^{2\lambda_j}dx),$ provided that $1\leq i\leq k< n.$\\

   From \eqref{Z0} and \eqref{2.27} we get for $\alpha>-1/2$
   \begin{equation}\label{2.39}
      \left| \frac{\partial}{\partial t}W_t^\alpha (u,v)\right| \leq C \left( (uv)^{-\alpha}\frac{e^{-(u-v)^2/10t}}{t^{3/2}}+ \frac{e^{-(u^2+v^2)/10t}}{t^{\alpha+ 3/2}} \right), \quad t,u,v \in (0,\infty).
   \end{equation}
   Suppose now that $1\leq k<i \leq n.$ By using \eqref{I1}, \eqref{I2}, \eqref{2.33} and \eqref{2.39} we obtain
   \begin{align*}
      \left|S_{\lambda_1,\ldots,\lambda_n}^{k,n,i} (g)(x)\right| \leq & C \int_0^\frac{x_1}{2} \ldots \int_0^\frac{x_k}{2} \int_\frac{x_{k+1}}{2}^{2x_{k+1}}\ldots \int_\frac{x_{n}}{2}^{2x_{n}}\int_0^\infty \frac{e^{-\overset{k}{\underset{j=1}{\sum}} x_j^2 /20t}}{t^{\overset{k}{\underset{j=1}{\sum}} (\lambda_j +1/2)}}\nonumber\\
      & \times \prod_{j=k+1,j\neq i}^n \left( \frac{1}{x_j^{2\lambda_j+1}} + (x_jy_j)^{-\lambda_j}\frac{e^{(x_j-y_j)^2/10t}}{\sqrt{t}}\right)\left( (x_iy_i)^{-\lambda_i}\frac{e^{-(x_i-y_i)^2/10t}}{t^{3/2}}\right. \\
      & +\left. \frac{e^{-(x_i^2+y_i^2)/4t}}{t^{\lambda_i+ 3/2}} \right)  dt |g(y)| \prod_{j=1}^n y_j^{2\lambda_j}dy\\
      & \leq C \int_0^\frac{x_1}{2} \ldots \int_0^\frac{x_k}{2} \int_\frac{x_{k+1}}{2}^{2x_{k+1}}\ldots \int_\frac{x_{n}}{2}^{2x_{n}}\int_0^\infty \frac{e^{-\overset{k}{\underset{j=1}{\sum}} x_j^2 /20t}}{t^{\overset{k}{\underset{j=1}{\sum}} (\lambda_j +1/2)+1}}\nonumber\\
      & \times \prod_{j=k+1}^n \left( \frac{1}{x_j^{2\lambda_j+1}} + (x_jy_j)^{-\lambda_j}\frac{e^{(x_j-y_j)^2/10t}}{\sqrt{t}}\right) dt |g(y)| \prod_{j=1}^n y_j^{2\lambda_j}dy, \quad x\in (0,\infty)^n.
   \end{align*}
   Then, as above, $S_{\lambda_1,\ldots,\lambda_n}^{k,n,i}$ is bounded from $L^p((0,\infty)^n, \overset{n}{\underset{j=1}{\prod}} x_j^{2\lambda_j}dx)$ into itself, for every $1<p<\infty,$ and from $L^1((0,\infty)^n, \overset{n}{\underset{j=1}{\prod}} x_j^{2\lambda_j}dx)$ into $L^{1,\infty}((0,\infty)^n, \overset{n}{\underset{j=1}{\prod}} x_j^{2\lambda_j}dx),$ when $1\leq k< i\leq n.$\\

   By proceeding in a similar way we can see that the operator $S_{\lambda_1,\ldots,\lambda_n}^{k,l,i}$ is bounded from $L^p((0,\infty)^n, \overset{n}{\underset{j=1}{\prod}} x_j^{2\lambda_j}dx)$ into itself, for every $1<p<\infty,$ and from $L^1((0,\infty)^n, \overset{n}{\underset{j=1}{\prod}} x_j^{2\lambda_j}dx)$ into $L^{1,\infty}((0,\infty)^n, \overset{n}{\underset{j=1}{\prod}} x_j^{2\lambda_j}dx),$ when $0\leq k<l<n,$ and $i=1,\ldots, n.$\\

   By \eqref{2.31.2} we conclude that $S_{\lambda_1,\ldots,\lambda_n}^{k,l},$ $0\leq k\leq l\leq n,$ $k\neq 0$ or $l\neq n,$ is bounded from $L^p((0,\infty)^n, \overset{n}{\underset{j=1}{\prod}} x_j^{2\lambda_j}dx)$ into itself, for every $1<p<\infty,$ and from $L^1((0,\infty)^n, \overset{n}{\underset{j=1}{\prod}} x_j^{2\lambda_j}dx)$ into $L^{1,\infty}((0,\infty)^n, \overset{n}{\underset{j=1}{\prod}} x_j^{2\lambda_j}dx).$ Then, the operator $\mathcal{G}_{\lambda_1,\ldots,\lambda_n}$ has the same $L^p$-boundedness properties.\\

   From \eqref{2.22} it follows that the maximal operator $T_{\lambda_1,\ldots,\lambda_n}^{m,*}$ is bounded from $L^p((0,\infty)^n, \overset{n}{\underset{j=1}{\prod}} x_j^{2\lambda_j}dx)$ into itself, for every $1<p<\infty,$ and from $L^1((0,\infty)^n, \overset{n}{\underset{j=1}{\prod}} x_j^{2\lambda_j}dx)$ into $L^{1,\infty}((0,\infty)^n, \overset{n}{\underset{j=1}{\prod}} x_j^{2\lambda_j}dx).$\\

   Thus the proof of this proposition is completed.
   \end{proof}

   From Proposition \ref{maximal} we can deduce by using standard arguments the following result.
   \begin{Prop}\label{limit}
     Let $\lambda_j >-1/2$, $j=1, \dots, n$. For every $f \in L^p((0,\infty)^n, \overset{n}{\underset{j=1}{\prod}} x_j^{2\lambda_j }dx),$ $1 \leq p < \infty,$
     the limit
     $$\lim_{\varepsilon \rightarrow 0^+} \int_{(0,\infty)^n, \ \lvert x-y \rvert>\varepsilon}
       f(y) K_{\lambda_1, \dots, \lambda_n}^\phi (x,y) \prod_{j=1}^n y_j^{2\lambda_j} dy, $$
     exists, for almost all $x \in (0,\infty)^n$. Here $K_{\lambda_1, \dots, \lambda_n}^\phi (x,y)$ is defined as in Proposition~\ref{maximal}. Moreover,
     the operator $\mathbb{T}_{\lambda_1, \dots, \lambda_n}^m$ defined by
     $$\mathbb{T}_{\lambda_1, \dots, \lambda_n}^m (f)(x)= \lim_{\varepsilon \rightarrow 0^+} \int_{(0,\infty)^n, \ \lvert x-y \rvert>\varepsilon}
       f(y) K_{\lambda_1, \dots, \lambda_n}^\phi (x,y) \prod_{j=1}^n y_j^{2\lambda_j} dy, \ \text{a.e. } x \in (0,\infty)^n, $$
     is bounded from $L^p((0,\infty)^n, \overset{n}{\underset{j=1}{\prod}} x_j^{2\lambda_j }dx)$ into itself,
     for every $1<p<\infty,$ and from $L^1((0,\infty)^n,$ $ \overset{n}{\underset{j=1}{\prod}} x_j^{2\lambda_j }dx)$ into
     $L^{1,\infty}((0,\infty)^n, \overset{n}{\underset{j=1}{\prod}} x_j^{2\lambda_j }dx).$
   \end{Prop}

   Since $C_c^\infty((0,\infty)^n)$ is a dense subspace of $L^p((0,\infty)^n, \overset{n}{\underset{j=1}{\prod}} x_j^{2\lambda_j}dx),$ $1\leq p<\infty,$ it follow that,
    for every $f\in L^2((0,\infty)^n, \overset{n}{\underset{j=1}{\prod}} x_j^{2\lambda_j}dx),$
   $$T_{\lambda_1, \dots, \lambda_n}^m (f)(x)= -\lim_{\varepsilon \rightarrow 0^+} \Big(\int_{(0,\infty)^n, \ \lvert x-y \rvert>\varepsilon}
       f(y) K_{\lambda_1, \dots, \lambda_n}^\phi (x,y) \prod_{j=1}^n y_j^{2\lambda_j} dy+\alpha(\varepsilon)f(x)\Big), \ \text{a.e. } x \in (0,\infty)^n, $$
   where $\alpha$ is a bounded function on $(0,\infty)$, and $T_{\lambda_1,\ldots,\lambda_n}^m$ can be extended from $L^2((0,\infty)^n, \overset{n}{\underset{j=1}{\prod}} x_j^{2\lambda_j}dx)$ $\bigcap L^p((0,\infty)^n, \overset{n}{\underset{j=1}{\prod}} x_j^{2\lambda_j}dx)$ to $L^p((0,\infty)^n, \overset{n}{\underset{j=1}{\prod}} x_j^{2\lambda_j}dx)$ as a bounded operator from $L^p((0,\infty)^n, $ $\overset{n}{\underset{j=1}{\prod}} x_j^{2\lambda_j}dx)$ into itself, for every $1<p<\infty,$ and from $L^1((0,\infty)^n, \overset{n}{\underset{j=1}{\prod}} x_j^{2\lambda_j}dx)$ into $L^{1,\infty}((0,\infty)^n,$ $\overset{n}{\underset{j=1}{\prod}} x_j^{2\lambda_j}dx).$\\

   The proof of Theorem \ref{multiplier} is finished.

\end{document}